\documentclass[letterpaper,11pt,reqno]{amsart}
\makeatletter
\usepackage{amssymb}
\usepackage{latexsym}
\usepackage{amsbsy}
\usepackage{amsfonts}
\usepackage{hyperref}
\usepackage{graphicx}
\usepackage{enumerate}
\usepackage{enumitem}
\usepackage{scalerel}
\usepackage{color}
\usepackage[round]{natbib}
\usepackage{comment}
\usepackage{scalerel}
\usepackage{tikz}

\def\marginpar#1{\ignorespaces}

\DeclareMathOperator\argmin{argmin}
\DeclareMathOperator\tr{Tr}

\newtheorem{theorem}{Theorem}
\newtheorem{lemma}[theorem]{Lemma}

\newtheorem{corollary}[theorem]{Corollary}
\newtheorem{definition}[theorem]{Definition}
\newtheorem{question}[theorem]{Question}
\newtheorem{assumption}[]{Assumption}

\newcommand{\lb}{\label}
\newcommand{\beq}{\begin{equation}}
\newcommand{\eeq}{\end{equation}}

\newcommand{\bal}{\begin{align}}
\newcommand{\eal}{\end{align}}
\newcommand{\bals}{\begin{align*}}
\newcommand{\eals}{\end{align*}}

\newcommand{\bbR}{{\mathbb{R}}}

\newcommand{\calS}{{\mathcal S}}

\newcommand{\calC}{{\mathcal C}}

\newcommand{\calU}{{\mathcal U}}

\newcommand{\eps}{\varepsilon}

\makeatother
\begin{document}
\title[Exploratory HJB equations]{Exploratory HJB equations and their convergence}

\author[Wenpin Tang]{{Wenpin} Tang}
\address{Department of Industrial Engineering and Operations Research, Columbia University.}
\email{wt2319@columbia.edu}

\author[Yuming Paul Zhang]{{Yuming Paul} Zhang}
\address{Department of Mathematics, University of California, San Diego.}
\email{yzhangpaul@ucsd.edu}

\author[Xun Yu Zhou]{{Xun Yu} Zhou}
\address{Department of Industrial Engineering and Operations Research, Columbia University.}
\email{xz2574@columbia.edu}

\date{\today}
\begin{abstract}
We study the exploratory Hamilton--Jacobi--Bellman (HJB) equation arising from the entropy-regularized exploratory control problem, which was formulated by Wang, Zariphopoulou and Zhou (J. Mach. Learn. Res., 21, 2020) in the context of reinforcement learning in continuous time and space.
We establish the well-posedness and regularity of the viscosity solution to 
the equation, as well as the convergence of the exploratory control problem to the classical stochastic control problem when the level of exploration decays to zero.
We then apply the general results to the exploratory temperature control problem, which was introduced by Gao, Xu and Zhou (arXiv:2005.04057, 2020) to design an endogenous temperature schedule for simulated annealing (SA) in the context of non-convex optimization.
We derive an explicit rate of convergence for this problem as exploration diminishes to zero, and find that the steady state of the optimally controlled process exists, which is however 
neither a Dirac mass  on the global optimum  nor  a Gibbs measure.
\end{abstract}

\maketitle



\section{Introduction}

\quad Reinforcement learning (RL) is an active subarea of machine learning. The RL research has predominantly focused on Markov Decision Processes (MDPs) in discrete time and space; see \cite{SB18} for a systematic account of the theory and applications, as well as a detailed description of bibliographical and historical development of the field. \cite{WZZ20} are probably the first to formulate and develop
an entropy-regularized,  exploratory control framework for RL in continuous time with
continuous feature (state) and action (control) spaces.
In this framework, stochastic relaxed control, a measure-valued process, is employed to  represent exploration through randomization, capturing the notion of ``trial and error" which is the core of RL. Entropy of the control is incorporated explicitly as a regularization term in the objective function to encourage exploration, with a weight parameter $\lambda > 0$ on the entropy to gauge the
tradeoff between exploitation (optimization) and exploration (randomization).
This exploratory formulation has been extended to other settings and used to solve applied problems; see  e.g. \cite{GXZ220} and \cite{FJ20} to mean-field games, and \cite{WZ20} to   Markowitz mean--variance portfolio optimization.  \cite{GXZ20} apply the same formulation to  temperature control of Langevin diffusions arising from simulated annealing  for non-convex optimization. The problem itself is not directly related to RL; however the authors take the same idea of ``exploration through randomization" and invoke
exploratory controls to smooth out the highly unstable yet theoretically optimal bang-bang control. For more literature review on the exploratory control,
see \cite{Zhou21}.

\quad \cite{WZZ20} derive the following  Hamilton--Jacobi--Bellman (HJB) partial differential equation (PDE) associated with the exploratory control problem, parameterized by the weight parameter $\lambda>0$:
\begin{equation}
\label{eq:regellipticPDE}
\scaleto{-\rho v_{\lambda}(x) + \lambda \ln \int_{\mathcal{U}} \exp\bigg(\frac{1}{\lambda} \bigg[h(x,u) + b(x,u) \cdot \nabla v_{\lambda}(x)
+ \frac{1}{2}\tr(\sigma(x,u) \sigma(x,u)^T \nabla^2v_{\lambda}(x))\bigg] \bigg) du = 0}{23 pt}.
\end{equation}
 This equation, called
    the {\it exploratory HJB equation},  appears to be characteristically different from the  HJB equation corresponding to a classical stochastic control problem.
    Among other things, \eqref{eq:regellipticPDE} does not involve the supremum operator in the control variable typically appearing in a classical HJB equation. This is because the supremum is replaced by a distribution among controls  in the exploratory formulation. 
    \cite{WZZ20}  do not study this general equation in terms of its well-posedness (existence and uniqueness of the viscosity solution), regularity, stability in $\lambda$, and the convergence when $\lambda\rightarrow 0^+$.
 They do, however, solve the important linear--quadratic (LQ) case where the exploratory HJB equation can be solved explicitly, leading to the optimal distribution for exploration being a Gaussian distribution. \cite{WZ20} apply this result to a continuous-time Markowitz portfolio selection problem which is inherently LQ.

\quad The goal of this paper is to study the general exploratory HJB equations beyond
the LQ setting.
We first analyze  a class of elliptic PDEs under fairly general assumptions on the coefficients (Theorems \ref{T.2.7} and \ref{T.10}). The application of the general results obtained to the exploratory HJB equations allows us to identify the assumptions needed, to derive the well-posedness of viscosity solutions and their regularity, and to establish a connection between the exploratory control problem and the classical stochastic control problem (Theorems \ref{lem:generalkey} and \ref{thm:main0}).
More specifically to the last point, we show that as the exploration weight decays to zero, the value function of the former converges to that of the latter.
This result, which extends \cite{WZZ20} to the general setting, is important for RL especially in terms of finding the regret bound (or the cost of exploration as termed in \cite{WZZ20}). As a passing note, our analysis for the general class of fully nonlinear elliptic PDEs may be of independent interest to the PDE community.

\quad In the second part of this paper, we 
focus on a special exploratory HJB equation resulting from the exploratory temperature control problem of the Langevin diffusions. The latter problem was introduced by \cite{GXZ20} aiming at designing a state-dependent temperature schedule for simulated annealing (SA). To provide a brief background (see \cite{GXZ20} for more details),
one of the central problems in continuous optimization is to escape from saddle points and local minima, and to find 
a global minimum of a non-convex function $f : \mathbb{R}^d \to \mathbb{R}$.
Applying the SA technique to the gradient descent algorithm consists of adding a sequence of independent Gaussian noises, scaled by ``temperature" parameters 
controlling the level of noises.
The continuous version of the SA algorithm  is governed by the following stochastic differential equation (SDE):
\begin{equation}
\label{eq:SA}
dX_t = -\nabla f(X_t) dt + \sqrt{2 \beta_t} dB_t, \quad X_0 = x,
\end{equation}
where $(B_t, \, t \ge 0)$ is a $d$-dimensional Brownian motion, and the temperature schedule $(\beta_t, \, t \ge 0)$ is a stochastic process. 
If $\beta_t \equiv \beta$ is constant in time, then \eqref{eq:SA} is the well-known overdamped Langevin equation whose stationary distribution is the Gibbs measure $\mathcal{G}_{\beta}(dx) \propto \exp(-f(x)/\beta)dx$ ($f$ is called the landscape, and $\beta$  the temperature).

\quad 
When allowing $(\beta_t, \, t \ge 0)$ to be a stochastic process, we have naturally a  stochastic control problem in which one controls the dynamics \eqref{eq:SA}  through this temperature process in order to achieve the highest efficiency in optimizing $f$. \cite{GXZ20} find that
the optimal control of this problem is of bang-bang type: the temperature process  switches between two extremum points in the search interval.
Such a bang-bang solution is almost unusable in practice since it is highly sensitive to errors. 
Moreover, in the present paper we discover  that the optimal state process under the bang-bang control may even not be well-posed in dimensions $d \ge 3$ (Section \ref{s41}).
These observations support the entropy-regularized exploratory formulation of temperature control proposed by \cite{GXZ20}, not so much from a learning perspective, but from a desire of smoothing out the bang-bang control. 

\quad The results for the general exploratory HJB equations apply readily to the temperature control setting in terms of the well-posedness, regularity and convergence (Corollaries \ref{T.2.7'} and \ref{cor:verification}).
Moreover, due to the special structure of the controlled dynamics \eqref{eq:SA},
we are able to derive an explicit convergence rate of $\lambda\ln(1/\lambda)$ for the exploratory temperature control problem as $\lambda$ tends to zero (Theorem \ref{thm:cvrate}).
Finally, we consider the long time behavior of the associated optimally controlled process and show that it will not converge to the global minimum of $f$ nor any Gibbs measure with landscape $f$ (Theorem \ref{thm:stationary}).
The first property is indeed preferred from an exploration point of view because  exploration is meant to 
involve as many states as possible instead of focusing only on the single state of the
minimizer. The second property hints  the possibility of a more
variety of target measures other than Gibbs measures for SA.

\quad The remainder of the paper is organized as follows.
In Section \ref{s2}, we provide some background on the exploratory control framework and
present the corresponding exploratory HJB equation.
In Section \ref{s3}, we investigate the exploratory  HJB equation and establish general results in terms of its well-posedness, regularity and convergence. 
We also identify the value function of the exploratory control problem as the unique solution to the exploratory  HJB equation.
In Section \ref{s4}, we apply the general results to the exploratory temperature control problem, derive an explicit  convergence rate, 
and study the long time behavior of the associated optimal state process.
While the main focus of the paper is on problems in the infinite time horizon, in Section \ref{s5}  we discuss the case of a  finite time horizon. Finally, 
Section \ref{s6} concludes with a few open questions suggested. 

\section{Background and problem formulation}
\label{s2}

\quad In this section, we provide some background on the exploratory control problem  that is put forth  in \cite{WZZ20}.

\quad Below we collect some notations that will be used throughout this paper.
\begin{itemize}[label = {--}, itemsep = 3 pt]
\item
For $x, y \in \mathbb{R}^d$, $x \cdot y$ denotes the inner product between $x$ and $y$, $|x| = \sqrt{\sum_{i = 1}^d x_i^2}$ denotes the Euclidean norm of $x$, $B_R = \{x: |x| \le R\}$ denotes the Euclidean ball of radius $R$ centered at $0$, and $|x|_{\max} = \max_{1 \le i \le d} |x_i|$ denotes the max norm of $x$.
\item
For a square matrix $X=(X_{ij}) \in \mathbb{R}^{d \times d}$, $X^T$ denotes its transpose, $\tr(X)$  its trace, $|X|$  its spectral norm, and $|X|_{\max} = \max_{1 \le i, j \le d}|X_{ij}|$  its max norm. Moreover,
$\calS^d = \{X \in \mathbb{R}^{d \times d}: X^T = X\}$ denotes the set of $d \times d$ symmetric matrices with the spectral norm.
\item
Let $\mathcal{O} \subseteq \bbR^d$ be open. For a function $f: \mathcal{O} \to \mathbb{R}$,
 $\nabla f$, $\nabla^2 f$ and $\Delta f = \tr(\nabla^2 f)$ denote respectively its gradient, Hessian and Laplacian.
\item For a bounded function $f: \mathcal{O} \to \mathbb{R}$, $||f||_{L^{\infty}(\mathcal{O})} = \sup_{x \in \mathcal{O}} |f(x)|$ denotes the sup norm of $f$.
\item
A function $f \in \mathcal{C}^k(\mathcal{O})$, or simply $f \in \mathcal{C}^k$, if it is $k$-time continuously differentiable.
The $\mathcal{C}^k$ norm is
\begin{equation*}
||f||_{\mathcal{C}^k} = \max_{|\beta| \le k} \sup_{x \in \mathcal{O}} |\nabla^{\beta}f(x)|,
\end{equation*}
where $\nabla^{\beta} f(x) = \frac{\partial ^{|\beta|} f}{\partial x_1^{\beta_1} \cdots \partial x_d^{\beta_d}}(x)$ with $\beta = (\beta_1, \ldots, \beta_d) \in \mathbb{N}^d$ and $|\beta| = \sum_{i=1}^d \beta_i$.
\item
A function $f \in \mathcal{C}^{k,\alpha}(\mathcal{O})$, or simply $f \in \mathcal{C}^{k,\alpha}$ ($0 < \alpha < 1$),  if it is $k$-time continuously differentiable and its $k^{th}$ derivatives of $f$ are $\alpha$-H\"older continuous.
The $\mathcal{C}^{k, \alpha}$ norm is
\begin{equation*}
||f||_{\mathcal{C}^{k, \alpha}} = \max_{|\beta| \le k} \sup_{x \in \mathcal{O}} |\nabla^{\beta}f(x)| + \max_{|\beta| = k}\sup_{x \ne y \in \mathcal{O}}\frac{|\nabla^{\beta}f(x) - \nabla^{\beta}f(y)|}{|x-y|^\alpha}.
\end{equation*}
\item
For two probability measures $\mathbb{P}$ and $\mathbb{Q}$,
 $||\mathbb{P} -\mathbb{Q}||_{TV} = \sup_{A} |\mathbb{P}(A) -\mathbb{Q}(A)|$ denotes the total variation distance between $\mathbb{P}$ and $\mathbb{Q}$.
\end{itemize}

\subsection{Classical control problem}
\label{s21}
Let $(\Omega, \mathcal{F}, \mathbb{P}, \{\mathcal{F}_t\}_{t \ge 0})$ be a filtered probability space on which we define a $d$-dimensional $\mathcal{F}_t$-adapted Brownian motion $(B_t, \, t \ge 0)$.
Let $\mathcal{U}$ be a generic action/control space, and $u = (u_t, \, t \ge 0)$ be a control which is an $\mathcal{F}_t$-adapted process taking values in $\mathcal{U}$.

\quad The classical stochastic control problem is to control the state variable $X_t \in \mathbb{R}^d$, whose dynamics is governed by the SDE:
\begin{equation}
\label{eq:classicalS}
dX^u_t = b(X^u_t, u_t) dt + \sigma(X^u_t, u_t) dB_t, \quad X_0^u = x,
\end{equation}
where $b: \mathbb{R}^d \times \mathcal{U} \to \mathbb{R}^d$ is the drift, and $\sigma: \mathbb{R}^d \times \mathcal{U} \to \mathbb{R}^{d \times d}$ is the covariance matrix of the state variable.
Here the superscript `$u$' in $X^u_t$ emphasizes the dependence of the state variable on the control $u$.
The goal of the control problem is to maximize the total discounted reward, leading to the (optimal) value function:
\begin{equation}
\label{eq:classicalV}
v(x) = \sup_{u \in \mathcal{A}_0(x)} \mathbb{E}\left[ \int_0^{\infty} e^{-\rho t}h(X^u_t, u_t) dt \bigg| X^u_0 = x\right],
\end{equation}
where $h: \mathbb{R}^d \times \mathcal{U} \to \mathbb{R}$ is a reward function, $\rho > 0$ is the discount factor, and $\mathcal{A}_0(x)$ denotes the set of admissible controls which may depend on the initial state value $X^u_0 = x$.

\quad By a standard dynamic programming argument, the HJB equation associated with the  problem \eqref{eq:classicalV} is
\begin{equation}
\label{eq:HJBclassical}
-\rho v(x) + \sup_{u \in \mathcal{U}} \left[h(x,u) + b(x,u) \cdot \nabla v(x) + \frac{1}{2}\tr(\sigma(x,u) \sigma(x,u)^T \nabla^2v(x))\right] = 0.
\end{equation}

In the classical stochastic control setting, the functional forms of $h, b, \sigma$ are given and known.
It is known that a suitably smooth solution to the HJB equation \eqref{eq:HJBclassical} gives the value function \eqref{eq:classicalV}.
Further, the optimal control is represented as a deterministic mapping from the current state to the action/control space:
$u^{*}_t = u^{*}(X^{*}_t)$.
The mapping $u^{*}$ is called an optimal feedback control,
which is derived offline from the ``$\sup_{u \in \mathcal{U}}$'' term in \eqref{eq:HJBclassical}. This procedure of obtaining the optimal feedback control is called the verification theorem.
The corresponding optimally controlled process $(X^{*}_t, \, t \ge 0)$ is governed by the SDE:
\begin{equation}
\label{eq:optimalclassical}
dX^{*}_t = b(X_t^{*}, u^{*}(X^{*}_t))dt +  \sigma(X_t^{*}, u^{*}(X^{*}_t))dB_t, \quad X_0^{*} = x,
\end{equation}
provided that it is well-posed (i.e. it has a unique weak solution).
See e.g. \cite{YZ99, FS06} for detailed accounts of the classical stochastic control theory.

\subsection{Exploratory control problem}
\label{s22}
In the RL setting, the model parameters are unknown, i.e. the functions $h, b, \sigma$ are not known.
Thus, one needs to explore and learn the optimal controls through repeated trials and errors.
Inspired by this, \cite{WZZ20} model exploration by a probability distribution of controls $\pi = (\pi_t(\cdot), \, t \ge 0)$ over the control space $\mathcal{U}$ from which each trial is sampled.
The exploratory state dynamics is
\begin{equation}
\label{eq:expS}
dX^\pi_t = \widetilde{b}(X^\pi_t, \pi_t)dt +  \widetilde{\sigma}(X^\pi_t, \pi_t)dB_t, \quad X^{\pi}_0=x,
\end{equation}
where the coefficients $\widetilde{b}(\cdot, \cdot)$ and $\widetilde{\sigma}(\cdot, \cdot)$ are defined by
\begin{equation}
\label{eq:tildebsig}
\widetilde{b}(x, \pi):=\int_{\mathcal{U}} b(x,u) \pi(u)du, \quad \widetilde{\sigma}(x, \pi):=\bigg(\int_{\mathcal{U}} \sigma(x,u) \sigma(x,u)^T \pi(u)du\bigg)^{\frac{1}{2}},\;(x,\pi)\in \mathbb{R}^d\times \mathcal{P}(\mathcal{U}),
\end{equation}
with $\mathcal{P}(\mathcal{U})$ being the set of absolutely continuous probability density functions on $\mathcal{U}$.
The distributional control $\pi = (\pi_t(\cdot), \, t \ge 0)$ is also known as the relaxed control,
and a classical control $u = (u_t, \, t \ge 0)$ is a special relaxed control when $\pi_t(\cdot)$ is taken as the Dirac mass at $u_t$.

\quad The exploratory control problem is an optimization problem similar to \eqref{eq:classicalV} but under relaxed controls. Moreover,
to  encourage exploration, Shannon's entropy is added to the objective function as a regularization term:
\begin{equation}
\label{eq:regV}
v_{\lambda}(x) = \sup_{\pi \in \mathcal{A}(x)} \mathbb{E}\bigg[ \int_0^{\infty} e^{-\rho t} \bigg( \int_{\mathcal{U}} h(X^\pi_t, u) \pi_t(u) du - \lambda \int_{\mathcal{U}} \pi_t(u) \ln \pi_t(u) du  \bigg) dt \bigg| X^u_0 = x \bigg],
\end{equation}
where $\lambda > 0$ is a weight parameter controlling the level of exploration (also called the temperature parameter),
and $\mathcal{A}(x)$ is the set of admissible distributional  controls specified by the following definition.

\begin{definition}
\label{def:admissible}
We say a density-function-valued stochastic process $\pi = (\pi_t(\cdot), \, t \ge 0)$, defined on a filtered probability space $(\Omega, \mathcal{F}, \mathbb{P}, \{\mathcal{F}_t\}_{t \ge 0})$ along with a $d$-dimensional $\mathcal{F}_t$-adapted Brownian motion $(B_t, \, t \ge 0)$, is an admissible distributional (or exploratory) control, denoted by $\pi \in \mathcal{A}(x)$, if
\begin{enumerate}[itemsep = 3 pt]
\item[($i$)]
For each $t \ge 0$, $\pi_t(\cdot) \in \mathcal{P}(\mathcal{U})$ a.s.;
\item[($ii$)]
For any Borel set $A \subset \mathcal{U}$, the process $(t, \omega) \to \int_A \pi_t(u,\omega)du$ is $\mathcal{F}_t$-progressively measurable;
\item[($iii$)]
The SDE \eqref{eq:expS} has solutions on the same filtered probability space whose distributions are all identical.
\end{enumerate}
\end{definition}

\quad Now we quickly review a formal derivation of the solution to the exploratory control problem \eqref{eq:expS}--\eqref{eq:regV}, following \cite{WZZ20}.
Again, by a dynamic programming argument, the HJB equation to  \eqref{eq:expS}--\eqref{eq:regV} is
\begin{multline}
\label{eq:HJBreg}
-\rho v_{\lambda}(x) + \sup_{\pi \in \mathcal{P}(\mathcal{U})} \int_{\mathcal{U}} \bigg(h(x,u) + b(x,u) \cdot \nabla v_{\lambda}(x) \\
+ \frac{1}{2}\tr(\sigma(x,u) \sigma(x,u)^T \nabla^2v_{\lambda}(x)) - \lambda \ln \pi(u)\bigg) \pi(u)du = 0.
\end{multline}
Then, through the same verification theorem argument,  the optimal distributional control is obtained by solving the maximization problem in  \eqref{eq:HJBreg} with the constraints $\int_{\mathcal{U}} \pi(u) du = 1$ and $\pi(u) \ge 0$ a.e. on $\mathcal{U}$.
This yields the optimal feedback control:
\begin{equation}
\label{eq:regfeedbackC}
\pi^{*}(u,x) = \frac{\exp\left(\frac{1}{\lambda} \left[h(x,u) + b(x,u) \cdot \nabla v_{\lambda}(x) + \frac{1}{2}\tr(\sigma(x,u) \sigma(x,u)^T \nabla^2v_{\lambda}(x))\right] \right)}{\int_{\mathcal{U}} \exp\left(\frac{1}{\lambda} \left[h(x,u) + b(x,u) \cdot \nabla v_{\lambda}(x) + \frac{1}{2}\tr(\sigma(x,u) \sigma(x,u)^T \nabla^2v_{\lambda}(x))\right] \right) du},
\end{equation}
which is the Boltzmann distribution or a Gibbs measure with landscape
\[
\frac{1}{\lambda} \left[h(x,u) + b(x,u) \cdot \nabla v_{\lambda}(x) + \frac{1}{2}\tr(\sigma(x,u) \sigma(x,u)^T \nabla^2v_{\lambda}(x))\right].
\]
By injecting \eqref{eq:regfeedbackC} into \eqref{eq:HJBreg}, we get the nonlinear elliptic PDE \eqref{eq:regellipticPDE}, or  the {exploratory HJB equation}. Note that  this equation is parameterized by
the weight parameter $\lambda>0$.

\quad Applying the feedback control \eqref{eq:regfeedbackC} to the state dynamics \eqref{eq:expS}, we obtain the optimally controlled dynamics:
\begin{equation}
\label{eq:optimalreglam}
dX^{\lambda,*}_t = \widetilde{b}(X^{\lambda,*}_t, \pi^{*}(\cdot, X^{\lambda,*}_t))dt +  \widetilde{\sigma}(X^{\lambda,*}_t, \pi^{*}(\cdot, X^{\lambda,*}_t))dB_t,
\end{equation}
provided that it is well-posed,
 i.e. it has a weak solution which is unique in distribution.
This condition is satisfied if $b(\cdot,\cdot)$ and $\sigma(\cdot,\cdot)$ are measurable and bounded,
$x \to \sigma(x, \cdot)$ is continuous,
and $\sigma(\cdot,\cdot)$ is strictly elliptic in the sense that
$\sigma(\cdot,\cdot)\sigma(\cdot,\cdot)^T \ge \Lambda I$;
see e.g. \cite{SV79} for discussions on the well-posedness of SDEs.
The optimal distributional control is then $\pi_t^{\lambda,*}(\cdot)=\pi^{*}(\cdot, X^{\lambda,*}_t)$, $t \ge 0$.

\quad 
The exploratory HJB equation \eqref{eq:regellipticPDE} is a new type of PDE in control theory, which begs a number of questions.
The first question is, naturally, its well-posedness (existence and uniqueness) in certain sense.
The second question is its dependence and convergence in  $\lambda >0$.
In practice, this parameter is often set to be small.
Thus, we are interested in
the limit of the solution to \eqref{eq:regellipticPDE} as $\lambda \to 0^{+}$, along with its convergence rate.
We will answer these questions in the following two  sections.

\section{Analysis of the exploratory HJB equation}
\label{s3}

\quad In this section, we study the exploratory HJB equation \eqref{eq:regellipticPDE}  under some
general assumptions on the functions $h(\cdot,\cdot), b(\cdot, \cdot), \sigma(\cdot,\cdot)$. For a concise analysis it is advantageous to analyze the general fully nonlinear elliptic PDEs of the form
\beq\lb{eq.1}
F(\nabla^2v,\nabla v,v,x)=0\quad \text{ in }\bbR^d.
\eeq
\quad In Section \ref{s31} we recall a few results on general elliptic PDEs in bounded domains, and prove a comparison principle for viscosity solutions of sub-quadratic growth in $\bbR^d$.
We show in Section \ref{s32} that, under some continuity and growth conditions on the operator $F$, \eqref{eq.1} has a unique smooth solution among functions that have sub-quadratic growth in $\bbR^d$.
In Section \ref{s33}, we consider a sequence of operators $F_\lambda$ that converge locally uniformly to $F$, and derive a convergence rate of the corresponding solutions $v_\lambda$ as $\lambda\to 0^+$ to $v$.
The rate of convergence for not necessarily bounded solutions with general operators (in particular with possibly unbounded coefficients) is novel.
Finally in Section \ref{s34}, we specify the general PDE results to the exploratory HJB equation \eqref{eq:regellipticPDE},
and prove a convergence result for the exploratory control problem \eqref{eq:expS}--\eqref{eq:regV} as $\lambda \to 0^+$.

\subsection{General results on second order elliptic equations}
\label{s31}

The standard references for second order elliptic PDEs are \cite{GT83,CC95}.
Here we recall some definitions and useful results.

\quad Consider the general fully nonlinear equations \eqref{eq.1}.
We make the following assumptions on the operator $F: \mathcal{S}^d \times \mathbb{R}^d \times \mathbb{R} \times \mathbb{R}^d \to \mathbb{R}$.
\begin{enumerate}[itemsep = 3 pt]
\item[($a$)] 
$F$ is continuous in all its variables, and for each $r\geq 1$ there exist $\gamma_{r},\underline{\gamma_{r}}>0$ such that for any $x,y\in B_r$ and $(X,p,q,s)\in \calS^d\times\bbR^{2d}\times\bbR$,
\[
|F(X,p,s,x)-F(X,q,s,y)|\leq \gamma_{r}|x-y|(1+|p|+|q|{+|X|})+\underline{\gamma_r}|p-q|,
\]
\[
|F(0,0,0,x)|\leq \gamma_r.
\]
\item[($b$)] 
There exist $\Lambda_2>\Lambda_1>0$ such that for any $P\in\calS^d$ positive semi-definite, and any $(X,p,s,x)\in\calS^d\times\bbR^d\times\bbR\times\bbR^d$,
    \[
\Lambda_2 \tr(P)   \geq  F(X,p,s,x)-F(X+P,p,s,x)\geq \Lambda_1 \tr(P).
    \]
\item[($c$)] There exists $\rho>0$ such that for all $(X,p,x)\in \calS^d\times\bbR^d\times\bbR^d$ and $t\geq s$,
    \[
   F(X,p,t,x)-F(X,p,s,x)\geq \rho(t-s).
    \]
\end{enumerate}

\quad These assumptions are standard (see \cite{vis,user}), and guarantee the existence and uniqueness of the viscosity solution to the equation \eqref{eq.1} in a bounded domain with a Dirichlet boundary condition.
The proof is given by Perron's method and the comparison principle.
Note that there exist weaker conditions than the ones stated above to ensure the well-posedness of  \eqref{eq.1} in bounded domains;
however, assumptions ($a$)--($c$) are simpler and  sufficient for our purpose.

\quad Now we recall the definition of viscosity solutions to   \eqref{eq.1}.

\begin{definition}
Let $\Omega$ be an open set in $\bbR^d$.
\begin{enumerate}
\item[(i)]
We say an upper semicontinuous (resp. lower semicontinuous) function $v:\Omega\to \mathbb{R}$ is a subsolution (resp. supersolution) to \eqref{eq.1} if the following holds:
for any smooth function $\phi$ in $\Omega$ such that $v-\phi$ has a local maximum (resp. minimum) at $x_0\in \Omega$, we have
\[
{F}(\nabla^2 \phi,\nabla\phi,v(x_0),x_0)\leq 0
\]
\[
(\text{ resp. }\quad {F}(\nabla^2 \phi,\nabla\phi,v(x_0),x_0)\geq 0. )
\]
\item[(ii)]
We say a continuous function $v:\Omega\to \mathbb{R}$ is a (viscosity) solution to \eqref{eq.1} if it is both a subsolution and a supersolution.
\end{enumerate}
\end{definition}

\quad Throughout this paper, by a solution of a PDE we mean a {\it viscosity} solution unless otherwise stated.

\quad Assume that there are a set of functions defined on in $\Omega$: $\{v_\eps(x),\,\eps>0\}$. Recall the definition of half-relaxed limits:
\begin{equation}\label{half}
v^*(x):=\limsup_{\substack{ \Omega\ni x'\rightarrow x,\\ \eps\rightarrow 0}}v_\eps (x'), \quad v_*(x):=\liminf_{\substack{ \Omega\ni x'\rightarrow x,\\ \eps\rightarrow 0}}v_\eps(x').
\end{equation}
Clearly,  $v^*$ is upper semicontinuous and $v_*$ is lower semicontinuous.
It is  known that  sub and supersolutions are stable under the half-relaxed limit operations; see \cite{user}.
\begin{lemma}
\label{lem stability}
Let $\Omega\subseteq\bbR^d$ be open, $\{F_\lambda,\,\lambda>0\}$ be a set of operators satisfying the assumptions (a)--(c) with the same constants.
Suppose that $F_\lambda$ converges locally uniformly in all its variables to an operator $\bar{F}$ as $\lambda\to 0^+$.
Then
\begin{enumerate}
    \item[(i)] if $v_\lambda$ is a sequence of bounded subsolutions to $F_\lambda(\nabla^2v_\lambda, \nabla v_\lambda,v_\lambda,\cdot)\leq 0$ in $\Omega$ for some $\lambda\to 0^+$, then their upper half-relaxed limit $v^*$ is a subsolution to
    \[
    \bar{F}(\nabla^2v^*,\nabla v^*,v^*,\cdot)\leq 0\quad\text{ in }\Omega;
    \]
    \item[(ii)] if $v_\lambda$ is a sequence of bounded supersolutions to $F_\lambda(\nabla^2v_\lambda,\nabla v_\lambda,v_\lambda,\cdot)\geq 0$ in $\Omega$ for some $\lambda\to 0^+$, then their lower half-relaxed limit $v_*$ is a supersolution to
    \[
    \bar{F}(\nabla^2v_*,\nabla v_*,v_*,\cdot)\geq 0\quad\text{ in }\Omega.
    \]
\end{enumerate}
\end{lemma}

\quad Next we consider the regularity of solutions to  \eqref{eq.1}.
We need the following additional assumption on the operator $F$.
\begin{definition}\lb{df4}
We say that an operator $F=F(X,p,s,x)$ is concave in $X$, if for any $M,N\in\calS^d,p,x\in\bbR^d$, and $s\in\bbR$ we have
\[
-\frac{\partial^2 F(M,p,s,x)}{\partial M_{ij}\partial M_{kl}}N_{ij}N_{kl}\leq 0,
\]
where the derivative and the inequality are in the sense of distribution.
\end{definition}

\quad The following result concerns higher regularity of bounded solutions to concave operators; see e.g. \cite{CC95} and \cite{lian2020pointwise}. 
As a consequence, viscosity solutions to concave operators are classical solutions.

\begin{lemma}[Theorems 2.1 and 2.6, \cite{lian2020pointwise}] \lb{T.1.5}
Assume that $F=F(X,p,s,x)$ satisfies (a)--(c), and
let $R_2 > R_1 > 0$.
If $v$ is a bounded viscosity solution to $F(\nabla^2v,\nabla v,v,x)=0$ in $B_{R_2}$,
then $v$ is $\mathcal{C}^{1,\alpha}$ in $B_{R_1}$. Moreover if $F$ is concave in $X$, then $v$ is $\mathcal{C}^{2,\alpha}$ in $B_{R_1}$. The upper bounds for $||v||_{\mathcal{C}^{1,\alpha}(B_{R_1})}$ or $||v||_{\mathcal{C}^{2,\alpha}(B_{R_1})}$ depend only on the constants in  assumptions (a)--(c), $R_1,R_2$, and $\|v\|_{L^\infty(B_{R_2})}$.
\end{lemma}

\quad Finally, we prove a comparison principle for solutions to \eqref{eq.1}, where the operator $F$ is assumed to have a  certain sub-quadratic growth in $x$ in the whole domain $\mathbb{R}^d$.
This comparison principle will be used to prove the uniqueness of the solution to the exploratory HJB equation \eqref{eq:regellipticPDE} under some assumptions on $h(\cdot,\cdot), b(\cdot, \cdot), \sigma(\cdot,\cdot)$.
\begin{lemma}[Comparison principle in $\bbR^d$] \lb{L.cp}
Assume that $F$ satisfies (a)--(c) with $\underline{\gamma_r}>0$ such that
\beq\lb{c.3}
\limsup_{r\to\infty}{\underline{\gamma_r}}/{r}=0.
\eeq
Let  $v_1$ and $v_2$ be locally uniformly bounded and be, respectively, a  subsolution and a  supersolution to  \eqref{eq.1} in
$\bbR^d$ such that
\beq\lb{c.5}
\limsup_{|x|\to\infty}\frac{v_1(x)-v_2(x)}{|x|^2}\leq 0.
\eeq
Then $v_1\leq v_2$ in $\bbR^d$.
\end{lemma}

\quad Note that in this lemma, there is no requirement on $\gamma_r$. A proof of Lemma \ref{L.cp} relies on the following classical comparison principle for elliptic PDEs in a {\it bounded} domain.
\begin{lemma}[Comparison principle, Theorem III.1, \cite{vis}] \lb{T.1.2}
Let $\Omega\subseteq\bbR^d$ be a bounded open set, and assume (a)--(c) hold.
Let $u$ (resp. $v$) be a bounded subsolution (resp. supersolution) to \eqref{eq.1} in $\Omega$ such that
\[
\limsup_{x\in\partial\Omega}(u(x)-v(x))\leq 0. 
\]
Then $u\leq v$ in $\Omega$.
\end{lemma}

\begin{proof}[Proof of Lemma \ref{L.cp}]
It follows from \eqref{c.3} that there exists $C>0$ such that for all $r\geq 0$,
\beq\lb{c.4}
(C+r^2 )\rho\geq 2\underline{\gamma_{r}} r.
\eeq
Set $C':=C+2d\Lambda_2\rho^{-1}$, and for any small $\eps>0$, define
\[
v^\eps(x):=v_2(x)+\eps (C'+|x|^2).
\]
We claim that $v^\eps$ is a supersolution to \eqref{eq.1} in $\bbR^d$.
Indeed, assume that there is $\varphi\in \mathcal{C}^\infty(\bbR^d)$ such that $v^\eps-\varphi$ has a local minimum at $x_0\in\bbR^d$. Then $v-\varphi^\eps$ with $\varphi^\eps:=\varphi-\eps (C'+|x|^2)$ has a local minimum at $x_0$.
Using the facts that $v_2$ is a supersolution and $F$ satisfies (a)--(c), we get by \eqref{c.4} that
\begin{align*}
F(\nabla^2\varphi,\nabla\varphi,v^\eps(x_0),x_0)&\geq F(\nabla^2\varphi^\eps,\nabla\varphi^\eps,v_2(x_0),x_0) -2d\Lambda_2\eps+\rho (C'+|x_0|^2)\eps-2\underline{\gamma_{|x_0|}}|x_0|\eps\\
&\geq (C+|x_0|^2 )\rho\eps-2\underline{\gamma_{|x_0|}} |x_0|\eps\geq 0.
\end{align*}
Hence  $v^\eps$ is a supersolution.

\quad Next, due to \eqref{c.5}, there exists $R_\eps>0$ such that $v^\eps(x)\geq v_1(x)$ for all $|x|\geq R_\eps$. Therefore applying Lemma \ref{T.1.2} to $v_1,v^\eps$ with $\Omega=B_{R_\eps}$ yields
\[
v^\eps(x)\geq v_1(x)\quad\text{ for all }x\in B_{R_\eps}.
\]
Taking $\eps\to 0$ leads to $v_2\geq v_1$ in $\bbR^d$. 
\end{proof}

\quad The above proof of Lemma \ref{L.cp} follows rather standard lines.
The comparison principle (and the well-posedness) for unbounded solutions to nonlinear elliptic equations in unbounded domains do exist
 in the literature; see e.g. \cite{user, capuzzo2005alexandrov, koike2011comparison, armstrong2015viscosity}.
However, those results do not apply to the problem in which we are interested.
In particular, none of these results covers the cases of unbounded $b(\cdot, \cdot)$ and/or $F$ being inhomogeneous in $X$, inherent in  the exploratory control problem.


\subsection{Well-posedness and stability}
\label{s32}
In this subsection we prove the well-posedness of solutions of sub-quadratic growth to  \eqref{eq.1}.
We need some assumptions on $\gamma_r,\underline{\gamma_r}$.
Let $\gamma:(0,\infty)\to (0,\infty)$ be  $\mathcal{C}^2$. Setting $\gamma_r:=\gamma(r),\gamma_r':=\gamma'(r), \gamma_r'':=\gamma''(r)$, we assume that
\beq\lb{3222}
\gamma_r'\geq 0\quad\text{ and }\quad \limsup_{r\to\infty}\frac{\gamma_r}{r^2}+\frac{\gamma_r'}{r}+\frac{\gamma_r'+|\gamma_r''|}{\gamma_r}=0 .
\eeq
This $\gamma_r$ represents a rate of sub-quadratic growth. For instance, we can take $\gamma_r = C(1+r^{a})$ or $C(1+r^{a}\ln(1+r))$ with $a\in [0,2),C>0$.

\begin{theorem}\lb{T.2.7} The following hold:
\begin{enumerate}
    \item[(i)] Assume that (a)--(c) hold with $\gamma_r$ satisfying \eqref{3222} and $\underline{\gamma_r}$ satisfying
\beq\lb{3222'}
\limsup_{r\to\infty} (\underline{\gamma_r}- \gamma_r/r)<\infty.
\eeq
Then there exists a unique solution $v$ of sub-quadratic growth to \eqref{eq.1},
and $v$ is locally uniformly $\mathcal{C}^{1,\alpha}$.
Moreover, there exists $C>0$ such that for all $r\geq 1$,
\beq\lb{b2.1}
\sup_{x\in B_r}|v(x)|\leq C\gamma_r.
\eeq
\item[(ii)] Assume that there are operators $F_\lambda$ satisfying (a)--(c) uniformly with the above $\gamma_r,\underline{\gamma_r}$ for $\lambda\in (0,1)$, such that $F_\lambda\to F$ as $\lambda\to 0^+$ locally uniformly in all the variables.
Then the unique solution $v_\lambda$ to
\beq\lb{1.1'}
F_\lambda(\nabla^2v_\lambda,\nabla v_\lambda,v_\lambda,x)=0\quad\text{ in }\bbR^d
\eeq
is $\mathcal{C}^{1,\alpha}$,  satisfies \eqref{b2.1}, and $v_\lambda \to v$ locally uniformly  as $\lambda\to 0^{+}$.
\item[(iii)] If $F$ (or $F_\lambda$) is concave in $X$, then $v$ (or $v_\lambda$) is locally uniformly $\calC^{2,\alpha}$.
    \end{enumerate}
\end{theorem}
\begin{proof}

(i) 
With the comparison principle (Lemma \ref{L.cp}), we only need to produce a supersolution and a subsolution that have sub-quadratic growth at infinity, and invoke Perron's method.

\quad By ($a$)--($c$) and \eqref{3222'}, there exists a constant $C>0$ such that for any $x\in\bbR^d$, $(X,p)\in \calS^d\times\bbR^{d}$, and $s\geq 0$, if $r:=|x|\geq 1$, then
\beq\lb{e90}
 F(X,p,s,x)\geq \rho s-\gamma_r(1+|p|/r)-C(1+|X|);
\eeq
and if $r\in [0,1)$, then
\beq\lb{e91}
 F(X,p,s,x)\geq \rho s-C(1+|p|+|X|).
\eeq
Let $\phi\in \mathcal{C}^2([0,\infty))$ be a regularization of $r\to\gamma_{r}$ such that
\beq\lb{e92}
\phi'(0)=\phi''(0)=0,\,\phi'(\cdot)\geq 0,\, \phi(r)=\gamma_r\text{ for $r\geq 1$},\,\text{ and }
\limsup_{r> 0}\phi'(r)/r<\infty .
\eeq
Define $\bar{v}(x):=C_1+C_2\phi(|x|)$ for some $C_1,C_2>0$ to be determined.
For simplicity, below we drop $(x)$ and $(|x|)$ from the notations of $\bar{v}(x)$, $\phi(|x|)$, $\phi'(|x|)$ and $\phi''(|x|)$. For $|x|\geq 1$, we have from \eqref{e90},
\[
F(\nabla^2\bar{v},\nabla \bar{v},\bar{v},x)\geq \rho (C_1+C_2\phi)-\phi(1+{C_2\phi'}/{|x|})-C(1+C_2|\phi''|).
\]
It follows from \eqref{3222} that
\beq\lb{c5}
\phi'(r)/r+\phi(r)^{-1}\left(|\phi''(r)|+{\phi'(r)}\right)\to 0\quad\text{ as }r\to \infty.
\eeq
Therefore by picking $C_2$ and then $C_1$ to be sufficiently large, we obtain
\[
F(\nabla^2\bar{v},\nabla \bar{v},\bar{v},x)\geq 0.
\]
This inequality holds the same when $|x|<1$ by \eqref{e91} and \eqref{e92}; 
Similarly, one can show that $\underline{v}:=-\bar{v}$ is a subsolution. Clearly both $\bar{v}$ and $\underline{v}$ have at most sub-quadratic growth.
Thus by Perron's method and Lemma \ref{L.cp} (note that by  \eqref{3222}, $\underline{\gamma_r}$ satisfies \eqref{c.3}), we obtain the unique solution $v$ to \eqref{eq.1}, and $\underline{v}\leq v\leq \bar{v}$ yields \eqref{b2.1}. Finally $v\in \calC^{1,\alpha}$ follows from Lemma \ref{T.1.5}.

\quad (ii) The above argument also yields the unique solution $v_\lambda$ to \eqref{1.1'}, with $v_\lambda\in\mathcal{C}^{1,\alpha}$  satisfying \eqref{b2.1} for each $\lambda\in (0,1)$.
Let $v_*,v^*$ be defined as in Lemma \ref{lem stability}. Since $F_\lambda\to F$ locally uniformly, Lemma \ref{lem stability} yields that $v_*$ and $v^*$ are, respectively, a supersolution and a subsolution to  \eqref{eq.1}. As $v_*$ and $v^*$ have at most sub-quadratic growth, applying Lemma \ref{L.cp} yields
$
v_*\geq v^* $ in $\bbR^d$.
The other direction of the inequality holds trivially by definition; hence  $v_*= v^*$ which then equals the unique solution $v$ to \eqref{eq.1}. This shows  $v_\lambda\to v$ locally uniformly as $\lambda\to 0^+$.

\quad (iii) This follows readily from Lemma \ref{T.1.5}.
\end{proof}

\subsection{Rate of convergence}
\label{s33}

Recall that $|X|$ denotes the spectral norm for $X\in\calS^d$.
We make the following assumption on the difference between $F$ and $F_\lambda$:
\begin{itemize}
    \item[($d$)]
    There exists a continuous function $\omega_0:[0,\infty)^4\to [0,\infty)$ such that  for each $\lambda\geq 0$,  $\omega_0(\lambda,\cdot,\cdot,\cdot)$ is non-decreasing in all its variables, $\omega_0(0,\cdot,\cdot,\cdot)\equiv 0$, and for each $(X,p,s,x)\in\calS\times\bbR^d\times\bbR\times\bbR^d$ we have
\[\lb{3334}
 |F_\lambda(X,p,s,x)-F(X,p,s,x)|\leq \omega_0(\lambda,|X|,|p|,|x|).
\]
\end{itemize}



\quad In the remainder of this subsection, we derive a convergence rate of $v_\lambda\to v$ as $\lambda\to 0^+$, assuming that the Lipschitz norms of $v_\lambda$ and $v$ are not too large at $x \to \infty$.
To our best knowledge, this error estimate result in the general setting with possibly unbounded solutions in $\bbR^d$ is new.


\begin{theorem}\lb{T.10}


Let $C_0\geq 1,\eta\in[0,2)$,  $F,F_\lambda$ satisfy (a)--(d) with $\gamma_r=C_0(1+r^\eta)$, $\underline{\gamma_r}=C_0(1+r^{\eta-1})$, and $v$ and $v_\lambda$ be, respectively, the solutions to  \eqref{eq.1} and \eqref{1.1'}. Suppose for some $\alpha\geq 0$, we have for each $r\geq 1$,
\beq\lb{222}
|\nabla v(\cdot)|+ |\nabla v_\lambda(\cdot)|\leq C_0r^\alpha\quad\text{ in } B_r.
\eeq
Then there exist $A,C>0$ such that for all $\lambda\in (0,1)$ and $r\geq 1$, we have
\beq\lb{con1}
\sup_{x\in B_r}|v_\lambda(x)-v(x)|\leq \rho^{-1}\omega_0(\lambda,R^{c_1},R^{c_2},R^{c_3})+CR^{-c_4},
\eeq
where
\[
R:=Ar, \quad \eps:=({2-\eta})/{2},\quad c_2:=1+\eps,\quad c_3:=\max\{\alpha(1+\eps),\,1\},
\]
\[
c_4:=1+\min\{(1-\eta)(1+\eps),0\},\quad c_1:=(2\alpha+\eta)(1+\eps)+c_4.
\]


\end{theorem}
\begin{proof}

We will only show that $v$ cannot be too much larger than $ v_\lambda$ for $\lambda\in (0,1)$ in $B_r$; the proof for the other direction is almost identical.
From the assumption and Theorem \ref{T.2.7}, there is $C_1\geq C_0$ such that for all $r\geq1$, we have $\gamma_r\leq C_1r^\eta$, $\underline{\gamma_r}\leq C_1(1+r^{\eta-1})$, and
\beq\lb{338}
|v(\cdot)|+|v_\lambda(\cdot)|\leq C_1r^\eta\quad\text{ in }B_r.
\eeq
Then after writing
\beq\lb{33.7}
\delta_r:=\sup_{x\in B_r} ({v(x)-v_\lambda(x)}),
\eeq
for some $r\geq 1$,
 \eqref{338} yields $\delta_r\leq C_1 r^\eta$.

\quad Let $R_1:=Ar$ for some $A\geq 1$, and $R_2:=R_1^{1+\eps}$ with $\eps=\frac{2-\eta}{2}\in (0,1]$. We consider a radially symmetric, and radially non-decreasing function $\phi:\bbR^d\to [0,\infty)$ such that \beq\lb{3336}
\phi(\cdot)\equiv 0\text{ on $B_r$},\quad \phi(\cdot)\geq C_1R_2^\eta \text{ on $\partial B_{R_2}$},
\eeq
and for some $C=C(d)$,
\beq\lb{3335}
|\nabla\phi(x)|\leq C(1+\phi(x))/R_1,\quad |\nabla^2\phi(x)|\leq C (1+\phi(x))/(R_1r)
\eeq
for all $x\in B_{R_2}$.
A regularization of the map $x\to \exp\left(\max\{0,x-r\}/R_1\right)-1$ will do if $A$ is large enough depending only on $\eta,C_1$. With one fixed $A$, below we prove a finer bound of $\delta_r$ for all $r$ large enough and $\lambda\in (0,1)$.

\quad Due to \eqref{338} and \eqref{3336}, there exists $x_0\in B_{R_2}$ such that
\beq\lb{3.71}
v(x_0)-v_\lambda(x_0)-2\phi(x_0)=\sup_{x\in \bbR^d} \left(v(x)-v_\lambda(x)-2\phi(x)\right)=:\delta'\geq {\delta_r}.
\eeq
Similarly, for any $\beta\geq 1$, we can find $x_1,y_1\in B_{R_2}$ such that
\beq\lb{3.7}
\begin{aligned}
 v(x_1) - v_\lambda(y_1) &-\phi(x_1)-\phi(y_1)- \beta|x_1 - y_1|^2\\
 & = \sup_{
x,y\in \bbR^d}
\left(v(x) - v_\lambda(y) - \phi(x)-\phi(y)-\beta|x - y|^2\right)\\
&\geq v(x_0)-v_\lambda(x_0)-2\phi(x_0)=\delta'.
\end{aligned}
\eeq
If $\phi(x_1)\leq \phi(y_1)$, noting
\[
|v_\lambda(x_1)-v_\lambda(y_1)|\leq C_0R_2^\alpha|x_1-y_1|
\]
in view of \eqref{222}, we conclude from \eqref{3.71} and \eqref{3.7} that
\begin{align*}
\delta' \leq v(x_1) - v_\lambda(x_1)& -2\phi(x_1)+C_0R_2^{\alpha}|x_1-y_1|-{\beta}|x_1-y_1|^2\\
&\qquad\qquad \leq \delta'+C_0R_2^{\alpha}|x_1-y_1|-{\beta}|x_1-y_1|^2,
\end{align*}
which yields
\beq\lb{3.8}
|x_1-y_1|\leq C_0 R_2^{\alpha}/\beta.
\eeq
This estimate still holds if $\phi(x_1)\geq \phi(y_1)$ by the same argument.
Let us write $C_\phi:=\phi(x_1)+\phi(y_1)$. It follows from \eqref{3.7} that
\beq\lb{3.12}
 v(x_1)-v_\lambda(y_1)\geq C_\phi+\delta'.
\eeq
Since $(1+\eps)\eta\leq 2$,  \eqref{3.12} and \eqref{338} yield
$C_\phi\leq C_1 R_2^\eta \leq C_1 R_1^2 $.

\quad Now we proceed by making use of \eqref{3.7}. 
Since $v,v_\lambda$ are solutions to \eqref{eq.1} and \eqref{1.1'} respectively, the Crandall-Ishii lemma \cite[Theorem 3.2]{user} yields that there are matrices $X$, $Y\in\calS^d$ satisfying the following:
\beq\lb{3.1}
-(2\beta+|J|)I\leq
\begin{pmatrix}
X & 0\\
0 & -Y
\end{pmatrix}
\leq J+\frac{1}{2\beta}J^2,
\eeq
where
\[
J:=2\beta\begin{pmatrix}
I & -I\\
-I & I
\end{pmatrix},
\]
and
\begin{align}\lb{3.2}
    F(X+\nabla^2\phi(x_1),p_1,v(x_1),x_1)\leq 0\leq F_\lambda(Y-\nabla^2\phi(y_1),q_1,v_\lambda(y_1),y_1),
\end{align}
where
\[
p_1:=2\beta (x_1-y_1)+\nabla\phi(x_1),\quad q_1:=2\beta (x_1-y_1)-\nabla\phi(y_1).
\]
Using (c) and \eqref{3.2} yields 
\[
\begin{aligned}
 \rho(v(x_1)-v_\lambda(y_1))\leq  F_\lambda(Y-\nabla^2\phi(y_1),q_1,v(x_1),y_1)-F(X+\nabla^2\phi(x_1),p_1,v(x_1),x_1).
\end{aligned}
\]
Writing $Y':=Y-\nabla^2\phi(y_1)$ and $Z:=X-Y+\nabla^2\phi(x_1)+\nabla^2\phi(y_1)$, we conclude  from (a),(b),(d), and $x_1,y_1\in B_{R_2}$ that
\beq\lb{3.13}
\begin{aligned}
 \rho(v(x_1)-v_\lambda(y_1))
&\leq \omega_0(\lambda,|Y'|,|y_1|,|q_1|) +C_1 R_2^\eta|x_1-y_1|(1+|p_1|+|q_1|+|Y'|)\\
&\qquad\qquad +C_1 (1+R_2^{\eta-1})|p_1-q_1|+\Lambda_2\tr(Z)1_{Z\geq 0}+\Lambda_1\tr(Z)1_{Z\leq 0}.
\end{aligned}
\eeq
Then we apply \eqref{338}, \eqref{3335}, \eqref{3.8}, and $C_\phi\leq C_1R_1^2$ to obtain
\begin{align*}
|q_1|&\leq C(R_2^\alpha+C_\phi R_1^{-1})\leq C(R_2^\alpha+R_1),\\
|x_1-y_1|(1+|p_1|+|q_1|+|Y'|) &   \leq C(R_2^{2\alpha}+C_\phi R_2^\alpha R_1^{-1}+R_2^\alpha|Y'|)/\beta \\
&\leq C(R_2^{2\alpha}+C_\phi R_2^\alpha R_1^{-1}+R_2^\alpha|Y|)/\beta,\\
(1+R_2^{\eta-1})|p_1-q_1|&\leq C(1+R_2^{\eta-1})(1+C_\phi)R_1^{-1}\leq C(1+C_\phi)R_1^{-c_4},
\end{align*}
where $c_4:=1+\min\{(1-\eta)(1+\eps),0\}\in (0,1]$ by $\eps=\frac{2-\eta}{2}$, and $C=C(C_0,C_1)>0$.

\quad Notice that $X\leq Y$, and $-6\beta I\leq  Y\leq 6\beta I$  by \eqref{3.1}. Therefore \eqref{3335} implies for some $C=C(\Lambda_2)>0$,
\[
\Lambda_2\tr(Z)1_{Z\geq 0}+\Lambda_1\tr(Z)1_{Z\leq 0}\leq -\Lambda_1\tr(Y-X)+ 
C(1+C_\phi)R_1^{-1}.
\]
Moreover, it follows from  $\beta\geq1$, $R_1=Ar$ and $C_\phi\leq CR_1^2$ that for some $C=C(A)>0$,
\[
|Y'|\leq  |Y|+CC_\phi (R_1r)^{-1}\leq C\beta.
\]
Plugging the above estimates into \eqref{3.13} shows
\[
\begin{aligned}
\rho(v(x_1)-v_\lambda(y_1))&\leq \omega_0(\lambda,C\beta,R_2,C(R_2^\alpha+R_1))-\Lambda_1\tr(Y-X)+CR_2^{\alpha+\eta}|Y|/\beta\\
&\qquad\qquad +C(R_2^{2\alpha+\eta}/\beta+R_1^{-c_4})+CC_\phi(R_2^{\alpha+\eta}R_1^{-1}/\beta+R_1^{-c_4}).
\end{aligned}
\]
Notice that by  \citet[Lemma 3.1]{vis} and \eqref{3.1}, there is $C=C(d)>0$ such that
\[
|X|+|Y|\leq C\tr(Y-X).
\]
Therefore if $CR_2^{\alpha+\eta}\leq \Lambda_1\beta $, we obtain
\beq\lb{3.11}
\begin{aligned}
\rho(v(x_1)-v_\lambda(y_1))&\leq \omega_0(\lambda,C\beta,R_2,C(R_2^\alpha+R_1))\\
&\qquad\qquad +C(R_2^{2\alpha+\eta}/\beta+R_1^{-c_4})+CC_\phi(R_2^{\alpha+\eta}R_1^{-1}/\beta+R_1^{-c_4}).
\end{aligned}
\eeq


\quad Now we pick $\beta:=R_1^{c_1}$ with $c_1:=(1+\eps)(2\alpha+\eta)+c_4$. Then
\[
\Lambda_1\beta=R_1^{c_1}\geq CR_2^{\alpha+\eta}=R_1^{(1+\eps)(\alpha+\eta)}
\]
holds when $r\geq 1$ ($R_1=Ar$) is large enough. By \eqref{3.11}, there exist $C,C'>0$ depending only on $C_0,C_1$ and $\eta$ such that
\[
\begin{aligned}
\rho(v(x_1)-v_\lambda(y_1))\leq \omega_0\left(\lambda,CR_1^{c_1},R_1^{1+\eps},C(R_1^{\alpha(1+\eps)}+R_1)\right)+CR_1^{-c_4}+ C'C_\phi R_1^{-c_4}.
\end{aligned}
\]
Recall \eqref{3.12}.
Upon further assuming $Ar=R_1\geq (C'/\rho)^{1/c_4}$, we have
\[
\rho\delta_r\leq \rho\delta'\leq \omega_0\left(\lambda,CR_1^{c_1},R_1^{1+\eps},C(R_1^{\alpha(1+\eps)}+R_1)\right)+CR_1^{-c_4}.
\]
This leads to the desired conclusion with $A$ replaced by $CA$, where $A,C>0$ depend only on $d,\eta$, $C_0$, $C_1$, $\rho$.
\end{proof}

\subsection{Exploratory HJB equations: well-posedness and convergence}
\label{s34}
Now we apply the general PDE results established in the previous subsections to study
the well-posedness of the exploratory HJB equation \eqref{eq:regellipticPDE} for fixed $\lambda>0$, as well as the convergence of  the solution as 
$\lambda \to 0^{+}$.

\quad We assume that the control space $\mathcal{U}$ is a non-empty open subset of some Euclidian space $\mathbb{R}^l$, 
and let $\rho > 0$.
Consider the operator associated with the exploratory HJB equation \eqref{eq:regellipticPDE}:
\begin{equation}
\label{eq:genFlam}
F_\lambda(X,p,s,x):=
\rho s-\lambda\ln \int_{u\in \mathcal{U}}\exp\left(\frac{1}{\lambda}(h(x,u)+b(x,u)p+\tr(\sigma(x,u)\sigma(x,u)^T X))\right)du,
\end{equation}
and the operator associated with the classical HJB equation \eqref{eq:HJBclassical}:
\begin{equation}
\label{eq:genFsup}
F(X,p,s,x):=\rho s-\sup_{u\in \mathcal{U}}\left(h(x,u)+b(x,u)p+\tr(\sigma(x,u)\sigma(x,u)^T X)\right).
\end{equation}

We also make the following assumptions on the functions $h(\cdot,\cdot), b(\cdot,\cdot), \sigma(\cdot,\cdot)$.
\begin{assumption}
\label{assump:hbsig}

There are positive $\gamma_r,\underline{\gamma_r}\in \mathcal{C}^2(0,\infty)$ satisfying \eqref{3222} and \eqref{3222'} such that the following hold:
\begin{enumerate}
\item[(i)]
For each $r \ge 1$, $|h(\cdot, \cdot)|$ is bounded by ${\gamma_r}$ in $B_r\times \mathcal{U}$, and $|b(\cdot,\cdot)|$ is bounded by $\underline{\gamma_r}$ in $B_r \times \mathcal{U}$.
\item[(ii)]
For each $r \ge 1$ and all $u \in \mathcal{U}$, $h(\cdot,u),$ $b(\cdot, u)$ and $\sigma(\cdot,u)$ are uniformly Lipschitz continuous with Lipschitz bound $\gamma_r$ in $B_r$.
\item[(iii)]
There exist $\Lambda_2>\Lambda_1>0$ such that $\Lambda_1I\leq \sigma(\cdot,\cdot)\sigma(\cdot,\cdot)^T\leq \Lambda_2I$ in $\bbR^d\times\mathcal{U}$.
\item[(iv)]
$h(\cdot,\cdot), b(\cdot,\cdot), \sigma(\cdot,\cdot)$ are locally uniformly continuous in $\mathbb{R}^d \times \mathcal{U}$.

\item[(v)] We have
\beq\lb{v1}
\sup_{\lambda\in (0,1)}\left| \lambda \ln\int_{u\in\calU}\exp\left(\frac{h(0,u)}{\lambda}\right)du \right|<\infty,
\eeq
and the following holds locally uniformly in $(X,p,x)\in \calS^d\times\bbR^d\times\bbR^d$:
\beq\lb{v2}
\limsup_{N\to\infty}\sup_{\lambda\in(0,1)}\left|\lambda \ln\int_{u\in\calU\backslash [-N,N]^l}\exp\left(\frac{1}{\lambda}(
h(x,u)+b(x,u)p+\tr(\sigma(x,u)\sigma(x,u)^T X
)\right)du\right|=0.
\eeq
\end{enumerate}
\end{assumption}

\quad The condition \eqref{v1} is to ensure that $F_\lambda$ with $\lambda\in (0,1)$ are well-defined, whereas
the condition \eqref{v2} is to guarantee  $F_\lambda\to F$ locally uniformly as $\lambda\to 0^+$ which is a reasonable requirement.
If $\calU$ is  a bounded set, then assumption $(v)$ holds trivially. Note that Assumption \ref{assump:hbsig} rules out the LQ case (i.e. $b(\cdot,\cdot), \sigma(\cdot,\cdot)$ are linear and $h(\cdot, \cdot)$ quadratic); but the corresponding exploratory and classical HJB equations for LQ can both be solved explicitly and the solutions are quadratic functions; see \cite{WZZ20}. In other words, the LQ case can be solved separately and specially and hence is not our concern here.

\smallskip

\quad We have the following result by specializing  the results in Subsections \ref{s32}--\ref{s33} to the operators $F_{\lambda}$, $F$ defined by \eqref{eq:genFlam}--\eqref{eq:genFsup}.
\begin{theorem}
\label{lem:generalkey}
Let $F_{\lambda}, F$ be defined by \eqref{eq:genFlam}--\eqref{eq:genFsup} and Assumption \ref{assump:hbsig} hold.
Then the assumptions (a)--(d) hold uniformly for $F_\lambda, F$ for all $\lambda\in(0,1)$, with
\[
\omega_0(\lambda,x_1,x_2,x_3):=\sup_{|X|\leq x_1,|p|\leq x_2, |x|\leq x_3} |F_\lambda(X,p,0,x)-F(X,p,0,x)|,
\]
and $F_{\lambda}, F$ are concave in $X$.
Consequently, the equation $F_{\lambda}(\nabla^2v_{\lambda}, \nabla v_{\lambda}, v_{\lambda}, x) = 0$ (resp. $F(\nabla^2v, \nabla v, v, x) = 0$) has a unique solution $v_{\lambda}$ (resp. $v$) of sub-quadratic growth.
Moreover,
\begin{enumerate}[itemsep = 3 pt]
\item[(i)]
$v_{\lambda}, v$ are locally $\mathcal{C}^{2,\alpha}$ for some $\alpha \in (0,1)$.
\item[(ii)]
There exists $C > 0$ such that $\sup_{B_r} |v(x)| + |v_{\lambda}(x)| \le C \gamma_r$ for each $r \ge 1$.
\item[(iii)]
$v_{\lambda} \to v$ locally uniformly as $\lambda \to 0^{+}$.
\end{enumerate}
\end{theorem}
\begin{proof}
It is direct to check that Assumption \ref{assump:hbsig} implies assumptions (a)--(c).
To see (d), note that if $\calU $ is a bounded set, $F_\lambda(X,p,s,x)\to F(X,p,s,x)$ locally uniformly in $X,p,s,x$ as $\lambda\to 0^+$ since $h(x,u),b(x,u),\sigma(x,u)$ are locally uniformly continuous in $u$ and uniformly continuous in $x$. If $\calU$ is unbounded, we use \eqref{v2} to get the convergence.

\quad Clearly the operator $F$ is concave in $X$ according to Definition \ref{df4}. Now we show that $F_\lambda$ is also concave in $X$. Let us write, for any fixed $p,x$,
\[
(a_{ij})=(a_{ij}(u)):=\sigma(x,u)\sigma(x,u)^T,\quad
g=g(X,u):=h(x,u)+b(x,u)p+\tr(\sigma(x,u)\sigma(x,u)^T X),
\]
and $G=G(X,u):=\exp(\lambda^{-1}g(X,u))$.
Then
\[
\frac{\partial g(X,u)}{\partial{X_{ij}}}=a_{ij}, \quad \frac{\partial^2g(X,u)}{\partial{X_{ij}}\partial{X_{kl}}}=0.
\]
Direct computation yields that for any $N=(N_{ij})\in\calS^d$,
\begin{align*}
-\frac{\partial^2 F_1(X,p,s,x)}{\partial X_{ij}\partial X_{kl}}N_{ij}N_{kl}&=\sum_{ij,kl}\frac{\left(\int_{\calU} G\,du\right)\left(\int_{\calU} a_{ij}a_{kl}G\,du\right)-\left(\int_{\calU} a_{ij} G\,du\right)\left(\int_{\calU} a_{kl}G\,du\right)}{\lambda\left( \int_{\calU}  G\,du\right)^2}N_{ij}N_{kl}\\
&=\frac{\left(\int_{\calU} G\,du\right)\left(\int_{\calU} (\sum_{ij}a_{ij}N_{ij} )^2 G\,du\right)-\left(\sum_{ij}\left(\int_{\calU} a_{ij}N_{ij} G\,du\right)\right)^2}{\lambda\left( \int_{\calU}  G\,du\right)^2}\geq 0,
\end{align*}
where the last inequality is due to H\"{o}lder's inequality and $G>0$. Therefore $F_\lambda$ is concave in $X$.
All the conclusions now follow from Theorem \ref{T.2.7}.
\end{proof}

\quad One can derive a convergence rate for $v_{\lambda} \to v$ as $\lambda \to 0^{+}$ in the spirit of Theorem \ref{T.10}, but we chose not to present it  in the above theorem because its expression would be overly complex for the general case.
In the next section, we will derive a simple,  explicit rate for a special application case -- the  temperature control problem.

\quad So far we have focused our attention on the HJB equations. 
The connection to the control problems is stipulated in the following theorem.
\begin{theorem}
\label{thm:main0}
Consider the exploratory control problem \eqref{eq:expS}--\eqref{eq:regV} with the value function $v_\lambda$.
Let Assumption \ref{assump:hbsig} hold, and assume that the SDE \eqref{eq:optimalreglam} is well-posed.
Then $v_{\lambda}$ is the unique solution of sub-quadratic growth to the exploratory HJB equation \eqref{eq:regellipticPDE}.
Moreover, $v_{\lambda}$ is locally $\mathcal{C}^{2,\alpha}$ for some $\alpha \in (0,1)$, and
\[
v_{\lambda} \to v \quad \mbox{locally uniformly as } \lambda \to 0^+,
\]
where $v$ is the value function of the classical control problem \eqref{eq:classicalS}--\eqref{eq:classicalV} and the unique solution of sub-quadratic growth to the classical HJB equation \eqref{eq:HJBclassical}.
\end{theorem}
\begin{proof}
Under Assumption \ref{assump:hbsig}, let $v'_{\lambda}$ be the unique solution to \eqref{eq:regellipticPDE}.
According to Theorem \ref{lem:generalkey} {(ii)}, $v'_{\lambda}$ has polynomial growth.
By a standard verification argument, we have $v_{\lambda}(x) \le v'_{\lambda}(x)$ for all $x \in \mathbb{R}^d$.
Since \eqref{eq:optimalreglam} is well-posed, the equality is achieved by the relaxed control $\pi^{*}_t(\cdot) = \pi^*(\cdot, X^{\lambda,*}_t)$, namely, $v_{\lambda} \equiv v'_{\lambda}$.
The remaining of the theorem follows readily from Theorem \ref{lem:generalkey}.
\end{proof}

\quad Theorem \ref{thm:main0} indicates that the exploratory control problem \eqref{eq:expS}--\eqref{eq:regV} converges to
the classical stochastic control problem \eqref{eq:classicalS}--\eqref{eq:classicalV} as the weight parameter $\lambda \to 0^{+}$.
The technical assumption needed is that
the optimally controlled process $(X^{\lambda,*}_t, \, t \ge 0)$ defined by the SDE \eqref{eq:optimalreglam} is well-posed.
If $\gamma_r=C(1+r)$ for some $C>0$ in Assumption \ref{assump:hbsig},
then it is easy to see that $x \to \widetilde{b}(x, \pi^*(\cdot, x))$ is bounded and measurable,
and $x \to \widetilde{\sigma}(\pi^*(\cdot, x))$ is bounded, continuous and strictly elliptic.
Classical theory of \cite{SV79} then implies that  \eqref{eq:optimalreglam} is well-posed.
However, if $\eta \in (1,2)$, then $b(\cdot,\cdot)$ and $x \to \widetilde{b}(x, \pi^*(\cdot, x))$ are unbounded.
Now, if it is true that the solution $v_{\lambda}$ to \eqref{eq:regellipticPDE} is locally $\mathcal{C}^3$ (under additional assumptions on $h(\cdot,\cdot), b(\cdot, \cdot), \sigma(\cdot,\cdot)$),
then we have that the functions $x \to \widetilde{b}(x, \pi^*(\cdot, x))$ and $x \to \widetilde{\sigma}(\pi^*(\cdot, x))$ are locally Lipschitz.
In this case, \eqref{eq:optimalreglam} has a unique strong solution, hence well-posed, up to the explosion time
$\tau_{\infty}:=\lim_{k \to \infty} \inf\{t \ge 0: |X^{\lambda,*}_t| > k\}$.
Further non-explosion conditions (see e.g. \cite{MT930}) ensure that $\tau_{\infty} = \infty$ almost surely, leading to the well-posedness of \eqref{eq:optimalreglam}.

\section{Exploratory temperature control problem}
\label{s4}

\quad In this section we apply the general results obtained in the previous section to the exploratory temperature control problem.
This problem was formulated by \cite{GXZ20} for temperature control in the context of SA.
In Section \ref{s41}, we provide a brief background
on this problem.
A detailed analysis of the associated exploratory HJB equation is given in Section \ref{s42}.
There we derive an explicit convergence rate for the value function as the weight parameter tends to zero.
Finally in Section \ref{s43}, we study the steady state of the optimally controlled process of the problem.

\subsection{Exploratory temperature control problem}
\label{s41}
To design an endogenous temperature control for SA,
\cite{GXZ20} first consider the following stochastic control problem:
\begin{equation}
\label{eq:bang}
\begin{aligned}
v(x) := & \inf \, \mathbb{E}\bigg[\int_0^\infty e^{-\rho t} f(X_t) dt \bigg], \\
    & \text{ subject to } \mbox{the equation } \eqref{eq:SA} \mbox{ where}\\
    &\qquad \{\beta_t, \, t \ge 0\} \mbox{ is adapted}, \mbox{and } \beta_t \in \mathcal{U} \mbox{ a.e. } t \ge 0, \mbox{ a.s.}.
\end{aligned}
\end{equation}
Here, the temperature process $(\beta_t, \, t \ge 0)$ is taken as the control.
Following \cite{GXZ20}, we take the control space $\mathcal{U} = [a,1]$ for a fixed $a \in (0,1)$ throughout this section.
Note that the upper bound of $\mathcal{U}$ can be replaced by any positive number, while we require that the lower bound of $\mathcal{U}$ be away from $0$ to guarantee a minimal effort for exploration.

\quad By setting $\mathcal{U} = [a,1]$, $h(x,u) = f(x)$, $b(x,u) = -\nabla f(x)$, $\sigma(x,u) = \sqrt{2u}$, and substituting ``$\sup$'' with ``$\inf$'' in \eqref{eq:HJBclassical}, we obtain the classical HJB equation of the temperature control problem \eqref{eq:bang}:
\begin{equation}
\label{eq:HJBbang}
-\rho v(x) + f(x) - \nabla f(x) \cdot \nabla v(x)+ \inf_{\beta \in [a,1]} \left[\beta \tr(\nabla^2v(x)) \right] = 0.
\end{equation}
It is then easily seen from the verification theorem that an optimal feedback control has the bang-bang form:
$\beta^{*} = 1$ if $\tr(\nabla^2v(x)) < 0$, and $\beta^{*} = a$ if $\tr(\nabla^2v(x)) \ge 0$.
Using this temperature control scheme, one should switch between the highest temperature and the lowest one, depending on the sign of $\tr(\nabla^2v(x))$.
As mentioned in the introduction, there are two disadvantages, one in theory and the other in application,  of this bang-bang strategy:
\begin{enumerate}[itemsep = 3 pt]
\item
Although theoretically optimal, this strategy is practically too rigid to achieve good performance as it only
has two actions: $a \to 1$ and $1 \to a$ .
It is too sensitive to  errors which are inevitable in any real world application.
\item
The corresponding optimally controlled dynamics is governed by the SDE:
\begin{equation}
\label{eq:switchSDE}
dX^*_t = - \nabla f(X^*_t) dt + g(X^*_t) dB_t, \quad X^*_0 = x,
\end{equation}
where
\begin{equation}
\label{eq:switchfct}
g(x) :=
\left\{ \begin{array}{lcl}
\sqrt{2a} & \mbox{if } \tr(\nabla^2v(x)) \ge 0, \\
\sqrt{2} & \mbox{if } \tr(\nabla^2v(x))< 0. \\
\end{array}\right.
\end{equation}
There is a subtle issue regarding the well-posedness of the SDE \eqref{eq:switchSDE}.
Note that $g$ is bounded and strictly elliptic.
If $\nabla f$ is assumed to be bounded, it follows from Exercise 12.4.3 in \cite{SV79} that
 \eqref{eq:switchSDE} has a weak solution for all dimension $d$.
However, the uniqueness in distribution may fail since $g$ is {\it discontinuous} (see e.g. \cite{Saf99} for an example).
According to Exercises 7.3.3 and 7.3.4 in \cite{SV79}, the uniqueness holds for $d = 1,2$.
But it remains unknown whether the uniqueness in distribution is still valid for $d \ge 3$.

There has been some literature on the uniqueness in distribution of SDEs with discontinuous diffusion coefficients via the martingale problem;
see e.g. \cite{BP87,ST97,Krylov04}.
In these works, it is assumed that the set of discontinuity has some special geometric structure.
However, for the diffusion coefficient \eqref{eq:switchfct}, the set of discontinuity is determined by the sign of $\tr(\nabla^2v(x))$, which is much more complex.
By Theorem \ref{T.2.7} below, $\nabla^2 v$ is continuous so the set $\{\tr(\nabla^2v) > 0\}$ (resp. $\{\tr(\nabla^2v) < 0\}$) is open;
but this condition alone cannot guarantee the uniqueness in distribution of  \eqref{eq:switchSDE}.
\end{enumerate}

\quad To address the first disadvantage above, \cite{GXZ20}  introduce the exploratory version of \eqref{eq:bang} in order to smooth out the temperature process.
This way, a classical control $(\beta_t, \, t \ge 0)$  is replaced by a relaxed control $\pi = (\pi_t(\cdot), \, t \ge 0)$ over the control space $\mathcal{U}= [a,1]$, rendering the following exploratory dynamics:
\begin{equation}
\label{eq:Xpi46}
dX^\pi_t = - \nabla f(X^\pi_t)dt + \bigg(\int_{\mathcal{U}} 2u \pi_t(u)du \bigg)^{\frac{1}{2}} dB_t.
\end{equation}
The exploratory temperature control problem is to solve
\begin{equation}
\label{eq:explore}
v_{\lambda}(x) := \inf_{\pi \in \mathcal{A}(x)} \mathbb{E}\bigg[\int_0^\infty e^{-\rho t} f(X^{\pi}_t) dt - \lambda \int_0^{\infty}e^{- \rho t} \int_{\mathcal{U}} - \pi_t(u) \ln \pi_t(u) dudt\bigg],
\end{equation}
where $\mathcal{A}(x)$ is the set of admissible controls specified by Definition \ref{def:admissible}.

\quad By setting $h(x,u) = f(x)$, $b(x,u) = -\nabla f(x)$ and $\sigma(x,u) = \sqrt{2u}$, 
we get the corresponding exploratory HJB equation:
\begin{equation}
\label{eq:ellipticPDE}
-\rho v_{\lambda}(x)  + \nabla f(x) \cdot \nabla v_{\lambda}(x)+ f(x) - \lambda \ln \int_a^1 \exp\left(- \frac{\tr(\nabla^2 v_{\lambda}(x))}{\lambda}  u \right) du = 0.
\end{equation}
The corresponding optimal feedback control is
\begin{equation}
\label{eq:feedbackpi}
\pi^*(u;x) = \frac{\exp\left(- \frac{\tr(\nabla^2 v_{\lambda}(x))}{\lambda}  u \right)}{\int_a^1 \exp\left(- \frac{\tr (\nabla^2 v_{\lambda}(x))}{\lambda}  u \right) du}, \quad u \in [a,1],
\end{equation}
which yields the optimally controlled process governed by the SDE:
\begin{equation}
\label{eq:Xlambdastar}
dX^{\lambda,*}_t = - \nabla f(X^{\lambda,*}_t) dt + g_{\lambda}(X^{\lambda,*}_t) dB_t,
\end{equation}
where
\begin{equation}
\label{eq:glambda}
g_{\lambda}(x) = \sqrt{2 \frac{\int_a^1 u \exp\left(- \frac{\tr(\nabla^2 v_{\lambda}(x))}{\lambda}  u \right) du}{\int_a^1 \exp\left(- \frac{\tr (\nabla^2 v_{\lambda}(x))}{\lambda}  u \right) du}},
\end{equation}

Note that the diffusion coefficient, $g_{\lambda}$, is now {\it continuous}, and $\sqrt{2a} \le g_{\lambda}(\cdot) \le 2$.
If $\nabla f$ is assumed to be bounded, it follows from the classical theory of \cite{SV79} that  \eqref{eq:Xlambdastar} is well-posed.
This is in stark contrast with the controlled dynamics \eqref{eq:switchSDE} which is not necessarily well-posed. In summary, the optimal temperature control scheme of this exploratory formulation allows any level of temperature and renders a well-posed state process, thereby remedying {\it simultaneously} the two aforementioned disadvantages of the classical formulation.

\quad To study the equation \eqref{eq:ellipticPDE} and the process governed by \eqref{eq:Xlambdastar},
we make the following assumptions on the function $f$.
\begin{assumption}
\label{assump:1}
The function $f \in \mathcal{C}^2$ satisfies
\begin{enumerate}[itemsep = 3 pt]
\item[($i$)]
There exists a constant $C > 0$ such that
\begin{equation*}
|\nabla f(x)| \le C \quad \mbox{and} \quad |\nabla^2 f(x)| \le C(1+|x|) \quad \mbox{for all } x \in \mathbb{R}^d.
\end{equation*}
\item[($ii$)]
There exist $\chi > 0$ and $R > 0$ such that
\begin{equation*}
|\nabla f(x)|^2 - d |\nabla^2f(x)|_{\max} \ge \chi \quad \mbox{for } |x| \ge R.
\end{equation*}
\end{enumerate}
\end{assumption}

\quad Note that a combination of $(i)$ and $(ii)$ yields a linear growth of $f$.
These conditions, in fact, guarantee that both the value function $v_{\lambda}$ and the optimal state process  $X^{\lambda,*}$ have good properties.
We will see that  $(i)$ alone is sufficient for identifying the value function $v_{\lambda}$ as the solution to the HJB equation, and
 $(ii)$ is essentially a Lyapunov/Poincar\'e condition which ensures the convergence of  $X^{\lambda,*}$ as $\lambda\rightarrow 0^+$.



\subsection{Analysis of the exploratory HJB equation}
\label{s42}

In this subsection, we apply the results in Section \ref{s3} to study \eqref{eq:ellipticPDE}.
The corresponding operators are
\beq\lb{Op:lambda}
F_\lambda(X,p,s,x):=\rho s  -\nabla f(x) \cdot p-f(x) +\lambda \ln \int_a^1 \exp\left(- \frac{\tr X}{\lambda}  u \right) du,
\eeq
and
\begin{equation}\lb{Op:F}
F(X,p,s,x):=\rho s-\nabla f(x)\cdot p-f(x)-(a1_{\tr X>0}+1_{\tr X<0})\tr X.
\end{equation}

Specializing Assumption \ref{assump:hbsig} to $\mathcal{U} = [a,1]$, $h(x,u) = f(x)$, $b(x,u) = -\nabla f(x)$ and $\sigma(x,u) = \sqrt{2u}$ leads to  the following assumption on $f$.
\begin{assumption}
\label{assump:2}
Assume that $f \in \mathcal{C}^2(\bbR^d)$, and there are positive $\gamma_r,\underline{\gamma_r}\in \mathcal{C}^2(0,\infty)$ such that for each $r\geq 1$,
\[
 \sup_{|x|<r} (|f(x)|+|\nabla^2 f(x)|)\leq \gamma_r\quad \text{ and }\quad  \sup_{|x|<r}|\nabla f(x)|\leq \underline{\gamma_r},
\]
where $\gamma_r,\underline{\gamma_r}$ satisfy \eqref{3222} and \eqref{3222'}. 
\end{assumption}

\quad Assumption \ref{assump:2} basically demands  a sub-quadratic growth on $f$ and a sub-linear growth on $|\nabla f|$.
It is more general than Assumption \ref{assump:1}-$(i)$.
In particular, it recovers  Assumption \ref{assump:1}$-(i)$ when $\gamma_r =C(1+r)$.

\quad The following result is an easy corollary of Theorem \ref{lem:generalkey}.

\begin{corollary}\lb{T.2.7'}
Let $F, F_{\lambda}$ be defined by \eqref{Op:lambda}--\eqref{Op:F}, and Assumption \ref{assump:2} hold. Then
\begin{enumerate}[itemsep = 3 pt]
\item[(i)]
There exists a unique solution $v$ of sub-quadratic growth to
the equation $F(\nabla^2 v, \nabla v, v, x)= 0$,
and $v$ is locally uniformly $\mathcal{C}^{2,\alpha}$.
\item[(ii)]
For each $\lambda>0$, there exists a
unique solution $v_\lambda$ of sub-quadratic growth to
the equation $F_{\lambda}(\nabla^2 v_{\lambda}, \nabla v_{\lambda}, v_{\lambda}, x) = 0$,
and $v_\lambda$ is locally uniformly $\mathcal{C}^{2,\alpha}$.
\item[(iii)]
There exists $C\geq 1$ such that  for all $r\geq 1$,
\beq\lb{b2.1'}
\sup_{\lambda\in (0,1)}\sup_{x\in B_r}\left(|v(x)|+|v_\lambda(x)|\right)\leq C(1+\gamma_r),
\eeq
and, moreover,  $v_\lambda\to v$ locally uniformly as $\lambda\to 0^+$.
\end{enumerate}
\end{corollary}

\quad Next we apply Theorem \ref{T.10} to derive an explicit rate of convergence for $v_\lambda\to v$ as $\lambda\to 0^+$, by
 assuming  that Assumption \ref{assump:2} holds with the choice of $\gamma_r=C(1+r^\eta)$ for some $\eta\in[0,2)$.

\begin{lemma}\lb{L.1.6}
Let Assumption \ref{assump:2} hold with $\gamma_r=C(1+r^\eta)$ for some $\eta\in[0,2)$.
Then
\begin{enumerate}[itemsep = 3 pt]
\item[(i)]
$F$ and $F_{\lambda}$ satisfy the assumptions $(a)$--$(c)$ with $\gamma_r=C(1+r^\eta), $ and $\underline{\gamma_r}=C(1+r^{\eta-1})$.
\item[(ii)]
The assumption $(d)$ holds with
\beq\lb{c.1}
\omega_0(\lambda,x_1,x_2,x_3):=\omega_0(\lambda,x_1)=C\lambda+\lambda\ln(dx_1/\lambda )1_{dx_1> \lambda},
\eeq
where $d$ is the dimesion of the state space.
\end{enumerate}
\end{lemma}

\begin{proof}
The proof of $(i)$ is the same as the one of Theorem \ref{lem:generalkey}, in which the expression of $\underline{\gamma_r}$ follows from \eqref{3222'}.
The proof of $(ii)$ follows from direct computations, and we will  prove \eqref{c.1} for  the case when $z:=\tr X/\lambda>0$ the other case being similar. Notice that
\[
A_\lambda:=F_\lambda(X,p,s,x)-F(X,p,s,x)=\lambda\ln\left[ z^{-1}\left(1-e^{-z(1-a)}\right)\right].
\]
If $z\geq 1$ we have
\[
z^{-1}\left(1-e^{-z(1-a)}\right)\in \left[z^{-1}(1-e^{-1+a}),z^{-1}\right],
\]
and if $z\in (0,1)$ we have
\[
z^{-1}\left(1-e^{-z(1-a)}\right)\in [1-e^{-1+a},1-a].
\]
Therefore
\[
|A_\lambda|\leq C\lambda+\lambda\ln(z)1_{z>1},
\]
and the conclusion follows since $d|X|\geq |\tr X|$.
\end{proof}

\quad In the following lemma, we present a point-wise bound of $|\nabla v| $ and $|\nabla v_\lambda|$.

\begin{lemma}\lb{L.reg}
Let Assumption \ref{assump:2} hold with $\gamma_r=C(1+r^\eta)$ for some $\eta\in[0,2)$.
Then there exists $C>0$ such that for any $r\geq 1$ we have
\[
\sup_{\lambda\in (0,1)}\sup_{x\in B_r}\left( |\nabla v(x)|+|\nabla v_\lambda(x)|\right)\leq C r^{\alpha}\quad\text{ where }\alpha:=\max\{2\eta-1,\eta\}.
\]
\end{lemma}
\begin{proof}
We will only prove for $v$, and that for $ v_\lambda$ is identical because $F_\lambda,\lambda>0$, have uniformly elliptic second order terms, while the lower order terms are the same as $F$.

\quad Fix  $r\geq 1$, and let
$
u(x):=r^{-\eta}{v(r^{-\gamma} x)}$ with $\gamma:=\max\{\eta-1,0\}$.
According to Corollary \ref{T.2.7'}, $u$ is uniformly bounded in $B_{2r^{1+\gamma}}$, and it satisfies
\[
\rho' u-b(x)\cdot \nabla u-c(x)-(a1_{\Delta u>0}+1_{\Delta u<0})\Delta u=0,
\]
where
\[
\rho':=\rho r^{-2\gamma},\quad b(x):=r^{-\gamma}(\nabla f)(r^{-\gamma}x),\quad c(x):=r^{-2\gamma-\eta}f(r^{-\gamma}x).
\]
Thus, by the assumption of the lemma and $\gamma\geq \eta-1$, we have for some $C>0$,
\[
\sup_{r\geq 1}\sup_{x\in B_{2r^{1+\gamma}}}\left(|b(x)|+|c(x)|\right)\leq C.
\]
This allows us to apply Theorem 2.1 in \cite{lian2020pointwise}  (see also Theorem 2.1 in \cite{swiech}) to conclude that $\sup_{x\in B_{r^{1+\gamma}}} |\nabla u(x)|\leq C$ for some $C$ independent of $r$, completing the proof.
\end{proof}


\quad Finally, we state the convergence rate result, the proof of which follows from Theorem \ref{T.10}, Lemma \ref{L.1.6} and Lemma \ref{L.reg}.

\begin{theorem}\label{thm:cvrate}
Let $F, F_{\lambda}$ be defined by \eqref{Op:lambda}--\eqref{Op:F}, and Assumption \ref{assump:2} hold with $\gamma_r=C(1+r^\eta)$ for some $\eta\in[0,2)$.
Also let $v_\lambda$ (resp. $v$) be the unique solution of sub-quadratic growth to the equation $F_{\lambda}(\nabla^2 v_{\lambda}, \nabla v_{\lambda}, v_{\lambda}, x) = 0$ (resp. $F(\nabla^2 v, \nabla v, v, x) = 0$).
Then there exists $C>0$ such that for all $\lambda\in (0,1)$ and $r\geq 1$ we have
\[
\sup_{x\in B_r}|v_\lambda(x)-v(x)|\leq C\lambda+C\lambda \ln(r/\lambda)+Cr^{-c}.
\]
with $c:=1+\min\left\{{(1-\eta)(4-\eta)}/{2},0\right\}$.
\end{theorem}

\quad Combining Theorem \ref{thm:main0} and Theorem \ref{thm:cvrate},
we get the following result characterizing the value function of the exploratory temperature control problem and its convergence.

\begin{corollary}
\label{cor:verification}
Consider the exploratory temperature control problem \eqref{eq:Xpi46}--\eqref{eq:explore} with value function $v_{\lambda}$.
Let Assumption \ref{assump:1}(i) hold.
Then $v_{\lambda}$ is the unique solution of sub-quadratic growth to the exploratory HJB equation \eqref{eq:ellipticPDE}.
Moreover, $v_{\lambda}$ is locally $\mathcal{C}^{2,\alpha}$ for some $\alpha \in (0,1)$, and
there exists $C>0$ such that for all $\lambda\in (0,1)$ and $r\geq 1$,
\begin{equation}
\label{eq:lambdaln}
\sup_{x\in B_r}|v_\lambda(x)-v(x)|\leq C\lambda+C\lambda \ln(r/\lambda)+Cr^{-1},
\end{equation}
where $v$ is the unique solution of sub-quadratic growth to the classical HJB equation \eqref{eq:HJBbang}.
\end{corollary}

\quad Because the constant $C>0$ in \eqref{eq:lambdaln} is independent of $\lambda\in (0,1)$ and $r\geq 1$, we can minimize the right hand side of \eqref{eq:lambdaln} with respect to $r$ to get  $r_{\min} = \lambda^{-1}>1$.
With $r_{\min}$, \eqref{eq:lambdaln} reduces to
\begin{equation}
\label{eq:lambdaln2}
\sup_{x\in B_{1/\lambda}}|v_\lambda(x)-v(x)|\leq 2 C\lambda+ 2 C\lambda \ln(1/\lambda).
\end{equation}
Note that for many real-world optimization problems,  one can (and probably {\it should}) restrict herself to a bounded set -- however large it might be --  containing all the ``important'' states.
Thus when  $\lambda$ is sufficiently small, the ball of radius $1/\lambda$ contains these states of interest, and the leading term on the right hand side of \eqref{eq:lambdaln2} is $\lambda \ln(1/\lambda)$. Therefore, the estimate \eqref{eq:lambdaln2} essentially stipulates  that $v_{\lambda}$ converges to $v$ at the rate of $\lambda \ln(1/\lambda)$ as $\lambda \to 0^{+}$.


\subsection{Optimally controlled state process}
\label{s43}
In this subsection we consider the long time behavior of the optimal state process \eqref{eq:Xlambdastar} of the exploratory temperature control problem.

\quad We start by recalling some basics in stochastic stability.
Consider the general diffusion process $X=(X_t, \, t \ge 0)$ in $\mathbb{R}^d$ of form:
\begin{equation}
\label{eq:diffusion}
dX_t = b(X_t) dt + \sigma(X_t) dB_t, \quad X_0 = x,
\end{equation}
where $b: \mathbb{R}^d \to \mathbb{R}^d$ is the drift, and $\sigma:  \mathbb{R}^d \to \mathbb{R}^{d \times d}$ is the diffusion (or covariance) matrix.
Assuming that  \eqref{eq:diffusion} is well-posed,
let $\mathcal{L}$ be the infinitesimal generator of the diffusion process $X$ defined by
\begin{equation}
\label{eq:infg}
\begin{aligned}
\mathcal{L}\psi(x) = \sum_{i=1}^d b_i(x) \frac{\partial}{\partial x_i} \psi(x) + \frac{1}{2} \sum_{i,j = 1}^d \left( \sigma(x) \sigma(x)^T \right)_{ij} \frac{\partial^2}{\partial x_i \partial x_j} \psi(x),
\end{aligned}
\end{equation}
and $\mathcal{L}^{*}$ be the corresponding adjoint operator given by
\begin{equation}
\begin{aligned}
\label{eq:adjoint}
\mathcal{L}^{*}\psi(x) = -\sum_{i = 1}^d \frac{\partial}{\partial x_i} (b_i(x) \psi(x)) + \frac{1}{2} \sum_{i,j = 1}^d \frac{\partial^2}{\partial x_i \partial x_j} (\sigma(x) \sigma(x)^T \psi(x))_{ij},
\end{aligned}
\end{equation}
where $\psi: \mathbb{R}^d \to \mathbb{R}$ is a suitably smooth test function.
The probability density $\rho_t(\cdot)$ of the process $X$ at time $t$ then satisfies the Fokker-Planck equation:
 \begin{equation}
 \label{eq:FPdiffusion}
 \frac{\partial \rho_t}{\partial t} = \mathcal{L}^{*} \rho_t.
 \end{equation}
It is not always true that $\rho_t(\cdot)$ converges as $t \to \infty$ to a probability measure.
But if $b$ and $\sigma$ satisfy some growth conditions, it can be shown that as $t\to \infty$,
$\rho_t(\cdot)$ converges in total variation distance to $\rho(\cdot)$  which is the stationary distribution (or steady state) of $X$.
 It is then easily deduced from \eqref{eq:FPdiffusion} that $\rho$ is characterized by the equation
$\mathcal{L}^{*} \rho = 0$.
For instance, the overdamped Langevin equation with $b(x) = - \nabla f(x)$ and $\sigma(x) = \sqrt{2 \beta} \, I$ is time-reversible, and the stationary distribution, under some growth condition on $f$, is the Gibbs measure
\begin{equation}
\label{eq:Gibbs}
\mathcal{G}_{\beta}(dx):= \frac{1}{Z_{\beta}} \exp \left(-\frac{f(x)}{\beta} \right)dx,
\end{equation}
where $Z_{\beta}:= \int_{\mathbb{R}^d} \exp(-f(x)/\beta)dx$ is the normalizing constant.
However, for general $b$ and $\sigma$, the stationary distribution $\rho(\cdot)$ may not have a closed-form expression.
The standard references for stability of diffusion processes are \cite{EK86, MT93, MT930}.
We record a result on the ergodicity of diffusion processes.
\begin{lemma}
 \label{lem:convstation}
 Assume that $b:\mathbb{R}^d \to \mathbb{R}^d$ is bounded, and $\sigma: \mathbb{R}^d \to \mathbb{R}^{d \times d}$ is bounded and strictly elliptic,
 and that there exists $0 <\alpha \le 1$ such that $b, \sigma$ are locally uniformly $\alpha$-H\"older continuous, i.e.
 for each $R > 0$ there is a constant $C_{R} > 0$ such that
\begin{equation}
\label{eq:partialLip}
|b(x) - b(y)| + |\sigma(x) - \sigma(y)| < C_{R} |x-y|^{\alpha} \quad \mbox{for all } x,y \in B_R.
\end{equation}
Then  \eqref{eq:diffusion} is well-posed, i.e. it has a weak solution which is unique in distribution.
Assume further that there exist $M_1 >0$, $M_2 < \infty$, a compact set $C \subset \mathbb{R}^d$, and a function $V: \mathbb{R}^d \rightarrow [1,\infty)$ with $V(x) \to \infty$ as $|x| \to \infty$ such that
 \begin{equation}
 \label{eq:Lyapbound}
 \mathcal{L}V \leq -M_1+M_2 1_{C},
 \end{equation}
Then the (unique) distribution of the solution to \eqref{eq:diffusion} converges in total variation distance to its unique stationary distribution as $t \to \infty$.
 \end{lemma}
 \begin{proof}
The fact that the diffusion process \eqref{eq:diffusion} is well-posed follows from Theorem 6.2 in \cite{SV79}.
Recall that a Borel set $C \subset \mathbb{R}^d$ is called petite if there exist a distribution $q$ on $\mathbb{R}_{+}$ and a nonzero Borel measure $\nu$ on $\mathbb{R}^d$ such that $\int_0^{\infty} \mathbb{P}_x(X_t \in A) \, q(dt) \ge \nu(A)$ for all $x \in C$ and all Borel sets $A \subset \mathbb{R}^d$.
Under the condition \eqref{eq:Lyapbound} with  a petite set $C$, Theorems 2.1 and 2.2 in \cite{Tang19} imply that
the diffusion process is positive Harris recurrent, and  converges in total variation distance to its unique stationary distribution.
Further by Theorem 2.1 in \cite{ST97}, the diffusion process is a Lebesgue irreducible (and $T$-) process.
However,  according to Theorem 4.1 in \cite{MT93}, each compact set is petite, which concludes the proof.
 \end{proof}

\quad The following theorem describes  the long time behavior of the optimal state process \eqref{eq:Xlambdastar} of the exploratory temperature control problem \eqref{eq:Xpi46}--\eqref{eq:explore}.
Recall that $||\cdot||_{TV}$ denotes the total variation distance between probability measures.
\begin{theorem}
\label{thm:stationary}
Let Assumption \ref{assump:1} hold.
Then we have:
\begin{enumerate}[itemsep = 3 pt]
\item[(i)]
For each $\lambda > 0$, the process $(X_t^{\lambda,*}, \, t \ge 0)$ converges in total variation distance to its unique stationary distribution as $t \to \infty$.
\item[(ii)]
For each $\lambda > 0$, let $\rho_{\lambda}$ be the stationary distribution of the process $(X_t^{\lambda,*}, \, t \ge 0)$.
Fix $\theta >0$ and $\delta>0$.
Then there exists $c > 0$ such that
\begin{equation*}
\rho_{\lambda}(\{x : |x-\theta| > \delta\}) > c \quad \mbox{for all } \lambda > 0.
\end{equation*}
Consequently, $(X_t^{\lambda,*}, \, t \ge 0)$ does not converge in probability to any $\theta \in \mathbb{R}^d$ (and in particular to $\argmin f(x)$).
\item[(iii)]
Let $\mathcal{G}_{\beta}$, $\beta > 0$, be the Gibbs measure of the form \eqref{eq:Gibbs}.
Then for each $\lambda > 0$, $\rho_{\lambda} \ne \mathcal{G}_{\beta}$ for any $\beta > 0$. Moreover, there exists $c > 0$ such that
\begin{equation*}
||\rho_{\lambda} - \mathcal{G}_{\beta}||_{TV} > c \quad \mbox{for all } \beta > 0.
\end{equation*}
\end{enumerate}
\end{theorem}

\begin{proof}
($i$) Note that $X^{\lambda,*}$ is a diffusion process with $b(x) = -\nabla f(x)$ and $\sigma(x) = g_\lambda(x) I$.
It is clear that $b$ is bounded, and $\sigma$ is bounded and strictly elliptic.
By Assumption \ref{assump:1}-($ii$), $|\nabla^2 f|$ is bounded, and thus $b = -\nabla f$ satisfies the H\"older condition \eqref{eq:partialLip}.
By Corollary \ref{T.2.7'}, $v_{\lambda}$ is locally $\mathcal{C}^2$.
It follows that $g_{\lambda}$ is locally H\"older continuous, and so is $\sigma = g_\lambda I$.
It is easy to see that
\begin{equation*}
\mathcal{L}f(x) =- |\nabla f(x)|^2 + \frac{1}{2}\sum_{i = 1}^d g^2_{\lambda}(x) \frac{\partial^2 f}{\partial x_i^2}(x) \le -|\nabla f(x)|^2 + d|\nabla^2 f(x)|_{\max}.
\end{equation*}
By Assumption \ref{assump:1}-($ii$), the condition \eqref{eq:Lyapbound} is satisfied with $M_1 = \chi$ and $M_2 = \sup_{x \in B_R} \mathcal{L}f(x)$.
It suffices to apply Lemma \ref{lem:convstation} to conclude.

\quad ($ii$) This follows from the fact that $g_{\lambda}$ is bounded away from $0$.
We argue by contradiction that
$\inf_{\lambda > 0} \rho_{\lambda}(\{x : |x-\theta| > \delta\}) = 0$.
Then for $\varepsilon > 0$, there exists $\lambda > 0$ such that $\rho_{\lambda}(\{x : |x-\theta| > \delta\}) < \varepsilon$.
By part ($i$), $(X_t^{\lambda,*}, \, t \ge 0)$ converges in total variation distance to $\rho_{\lambda}$.
So for $t$ sufficiently large, we have
\begin{equation}
\label{eq:uppsmall}
\mathbb{P}(X_t^{\lambda,*}> \theta+ \delta) < 2 \varepsilon.
\end{equation}
On the other hand, $b = -\nabla f$ and $\sigma = g_{\lambda} I$ are H\"{o}lder continuous, and $\sigma \sigma^T \ge 2a I$ with $2a$ independent of $\lambda$.
By Aronson's comparison theorem (see \cite{Aronson67}),
\begin{equation}
\label{eq:lowlarge}
\mathbb{P}(X_t^{\lambda,*}> \theta+ \delta) \ge C \mathbb{P}(cB_t > \theta + \delta),
\end{equation}
where $c, C > 0$ are constants independent of $t$ and $\lambda$.
By taking $\varepsilon > 0$ to be arbitrarily small, the estimates \eqref{eq:uppsmall} and \eqref{eq:lowlarge} lead to a contradiction.

\quad ($iii$) We first prove that $\rho_{\lambda} \ne \mathcal{G}_{\beta}$ for any $\beta > 0$.
We argue by contradiction that $\rho_{\lambda} = \mathcal{G}_{\beta}$ for some $\beta > 0$.
Recall from \eqref{eq:adjoint} that the adjoint operator of the optimal controlled process is
$\mathcal{L}^{*}\psi(x) = - \sum_{i = 1}^d \frac{\partial}{\partial x_i}\bigg(\frac{\partial f}{\partial x_i}(x) \psi(x) \bigg) + \frac{1}{2} \sum_{i = 1}^d \frac{\partial^2}{\partial x_i^2}(g_{\lambda}(x) \psi(x))$
for $\psi: \mathbb{R}^d \to \mathbb{R}$.
Since $\mathcal{L}^{*} \rho_{\lambda} = 0$, we get
\begin{equation}
\label{eq:FPOCP1}
- \sum_{i = 1}^d \frac{\partial}{\partial x_i}\bigg(\frac{\partial f}{\partial x_i}(x) \rho_{\lambda}(x) \bigg) + \frac{1}{2} \sum_{i = 1}^d \frac{\partial^2}{\partial x_i^2}(g_{\lambda}(x) \rho_{\lambda}(x)) = 0.
\end{equation}
On the other hand, $\rho_{\lambda} = \mathcal{G}_{\beta}$ is the stationary distribution of the overdamped Langevin equation $dX_t = - \nabla f(X_t) dt + \sqrt{2 \beta} dB_t$; so it satisifies
\begin{equation}
\label{eq:FPOCP2}
- \sum_{i = 1}^d \frac{\partial}{\partial x_i}\bigg(\frac{\partial f}{\partial x_i}(x) \rho_{\lambda}(x) \bigg) + \frac{\beta}{2} \sum_{i = 1}^d \frac{\partial^2}{\partial x_i^2} \rho_{\lambda}(x) = 0.
\end{equation}
Comparing \eqref{eq:FPOCP1} and \eqref{eq:FPOCP2} yields
\begin{equation*}
\Delta (g_{\lambda} \rho_{\lambda} - \beta \rho_{\lambda}) = 0,
\end{equation*}
i.e. $g_{\lambda} \rho_{\lambda} - \beta \rho_{\lambda}$ is a harmonic function.
By Assumption \ref{assump:1}-$(ii)$, $f(x) \to  +\infty$ as $|x| \to \infty$.
Thus, $g_{\lambda} \rho_{\lambda} - \beta \rho_{\lambda} \to 0$ as $|x| \to \infty$.
According to Liouville's theorem, any bounded harmonic function is constant (see e.g. Theorem 8, Chapter 2 in \cite{Evansbook}).
So $g_{\lambda} \rho_{\lambda} - \beta \rho_{\lambda}  \equiv 0$, and hence $g_{\lambda} \equiv \beta$.
Injecting to \eqref{eq:glambda}, we see that $v_{\lambda}$ only depends on $a$, $\beta$ and $\lambda$.
This contradicts the HJB equation \eqref{eq:ellipticPDE} where $v_{\lambda} $ also depends on $f$.

\quad Now we prove that $\rho_{\lambda}$ is bounded away from {\it any} Gibbs measure $\mathcal{G}_{\beta}$.
We argue by contradiction that $\inf_{\beta > 0} ||\rho_{\lambda} - \mathcal{G}_{\beta}||_{TV} = 0$.
Then there exists a sequence $\{\beta_n\}_{n \ge 1}$ such that $||\rho_{\lambda} - \mathcal{G}_{\beta_n}||_{TV} \to 0$ as $n \to \infty$.
This is impossible if $\lim_{n \to \infty} \beta_n = \infty$, since $\mathcal{G}_{\beta}$ does not converge to a probability measure as $\beta \to \infty$.
Thus, we can extract a convergent subsequence $\{\beta'_n\}_{n \ge 1}$ from $\{\beta_n\}_{n \ge 1}$.
If $\lim_{n \to \infty} \beta'_n = \beta' > 0$, this implies that $\rho_{\lambda} = \mathcal{G}_{\beta'}$ which contradicts the fact that
$\rho_{\lambda} \ne \mathcal{G}_{\beta}$ for any $\beta > 0$.
If $\lim_{n \to \infty} \beta'_n = 0$, then $\rho_{\lambda}$ is concentrated on $\argmin f$, whose validity  is ruled out by part $(ii)$.
\end{proof}

\quad Theorem \ref{thm:stationary} indicates that, with a {\it fixed} level of exploration, the optimally controlled process $(X_t^{\lambda,*}, \, t \ge 0)$ does have a stationary distribution. This provides a theoretical justification to the SA algorithm devised by \cite{GXZ20} based on discretizing \eqref{eq:Xlambdastar}. The result that this stationary distribution is
not a Dirac mass on the minimizer of $f$ is expected theoretically because  \eqref{eq:Xlambdastar} is a genuine diffusion process due to its strict ellipticity.
It is indeed {\it preferred} from an exploration point of view because the essence of exploration is to involve as many states as possible instead of just focusing on the
single state of the minimizer, in the same spirit of the classical overdamped Langevin diffusion that
converges to the Gibbs measure instead of the Dirac one. The fact that the stationary distribution of \eqref{eq:Xlambdastar} is {\it not} a Gibbs measure is the most intriguing one; it suggests the possibility of a more variety of target measures -- beyond Gibbs measures -- when it comes to SA for non-convex optimization.


\quad Theorem \ref{thm:stationary} does not provide a convergence rate for $(X_t^{\lambda,*}, \, t \ge 0)$ to converge to its stationary distribution.
This is due to the assumption that $|\nabla f|$ is bounded, which is a sufficient condition for the well-posedness of the process $(X_t^{\lambda,*}, \, t \ge 0)$ in the verification argument.
If we can relax this  condition while  the process $(X_t^{\lambda,*}, \, t \ge 0)$ is still well-posed,
then more can be said about the convergence.
For instance, assume that $|\nabla^2 f|$ is bounded, $g_{\lambda}$ is locally Lipschitz, and
\begin{equation}
\label{eq:subgeo}
|\nabla f(x)|^2 \ge \phi(f(x)) \quad \mbox{for } |x| \mbox{ sufficiently large},
\end{equation}
where $\phi$ is a strictly concave function increasing to infinity (e.g. $\phi(s) = s^{\alpha}$ for some $0 < \alpha < 1$).
In this case, let $H_{\phi}(s) = \int_1^s \frac{ds}{\phi(s)}$.
Then there exists $C > 0$ such that \citep{BCG08}
\begin{equation*}
|| \mbox{Law}(X^{\lambda,*}_t) - \rho_{\lambda}||_{TV} \le \frac{C}{\phi \circ H^{-1}_{\phi}(t)}.
\end{equation*}
So $(X_t^{\lambda,*}, \, t \ge 0)$ converges to its stationary distribution with a sub-exponential rate.
If instead of \eqref{eq:subgeo} we assume that
\begin{equation}
\label{eq:expconv}
|\nabla f(x)|^2 \ge M f(x) \quad \mbox{for } |x| \mbox{ sufficiently large},
\end{equation}
for some $M > 0$, then there exist $c > 0$ and $C >0$ such that
\begin{equation*}
|| \mbox{Law}(X^{\lambda,*}_t) - \rho_{\lambda}||_{TV} \le C e^{-ct}.
\end{equation*}
That is,  $(X_t^{\lambda,*}, \, t \ge 0)$ converges exponentially to its stationary distribution.
See e.g. \cite{BCG08} for further discussions on the convergence rate of diffusion processes.
This means that if we can relax the well-posedness condition (e.g. removing the boundedness assumption on $|\nabla f|$ so that either \eqref{eq:subgeo} or \eqref{eq:expconv} is satisfied),
then we can derive a convergence rate for the optimally controlled process $(X_t^{\lambda,*}, \, t \ge 0)$ of the exploratory temperature control problem as $t\to \infty$.

\quad To conclude this subsection, we study the stability of stationary distributions of $(X_t^{\lambda,*}, \, t \ge 0)$ with different $\lambda$'s.
For a general analysis on  the stability of stationary distributions of diffusion processes with different drift and covariance coefficients,
see \cite{BKS14,BKS17,BRS18}.
The idea is to bound the total variation distance between stationary distributions in terms of diffusion parameters.
We recall a lemma which is due to \cite{BRS18}.
\begin{lemma}
\label{lem:BRSest}
Let $(b_1, \sigma_1)$ and $(b_2, \sigma_2)$ be pairs of drift and covariance coefficients associated with the diffusion process \eqref{eq:diffusion}.
For each $k= 1,2$, assume that $b_k$ is bounded and measurable, and $\sigma_k$ is bounded, strictly elliptic and globally Lipschitz.
Then the diffusion process associated with $(b_k, \sigma_k)$ has a unique stationary distribution $\rho_k(dx) = \rho_k(x) dx$.
For $1 \le i \le d$, let
\begin{equation*}
\phi_1^i := b_1^i - \sum_{j = 1}^d \frac{\partial}{\partial x_j}(\sigma_1 \sigma_1^T)_{ij} \quad \mbox{and}
\quad \phi_2^i := b_2^i - \sum_{j = 1}^d \frac{\partial}{\partial x_j}(\sigma_2 \sigma_2^T)_{ij},
\end{equation*}
and $\Phi := \frac{(\sigma_1 \sigma_1^T- \sigma_2 \sigma_2^T) \nabla \rho_2}{\rho_2} - (\phi_1 - \phi_2)$.
Assume further that there exist $\kappa > 0$, $M > 0$ and $R > 0$ such that
\begin{equation*}
b_1(x) \cdot x \le -M |x|^{\kappa}, \quad \mbox{for } |x| > R.
\end{equation*}
Then there exists $C > 0$ such that
\begin{equation*}
||\rho_{1} - \rho_2||_{TV} \le C \int_{\mathbb{R}^d} |\Phi(x)| \rho_2(dx).
\end{equation*}
\end{lemma}

\begin{theorem}
\label{thm:stability}
Let Assumption \ref{assump:1} hold, and assume further that
there exist $\kappa > 0$, $M > 0$ and $R > 0$ such that
\begin{equation}
\label{eq:dissipative}
\nabla f(x) \cdot x \ge M |x|^{\kappa}, \quad \mbox{for } |x| \ge R,
\end{equation}
and that the solution $v_{\lambda}$ to  \eqref{eq:ellipticPDE} is $\mathcal{C}^3$ with bounded third derivatives.
For each $\lambda > 0$, let $\rho_{\lambda}(dx)$ be the stationary distribution of the optimal state process governed by \eqref{eq:Xlambdastar}.
Then
\begin{equation}
\lim_{\lambda' \to \lambda}||\rho_{\lambda'} - \rho_{\lambda}||_{TV} = 0.
\end{equation}
\end{theorem}
\begin{proof}
We apply Lemma \ref{lem:BRSest} with $b_1(x) = b_2(x) = -\nabla f(x)$, and $\sigma_1(x) = g_{\lambda'}(x) I$, $\sigma_2(x) = g_{\lambda}(x)I$.
In this case,
\begin{equation*}
\Phi(x) = (g_{\lambda'}(x) - g_{\lambda}(x)) \frac{\nabla \rho_{\lambda}(x)}{\rho_{\lambda}(x)} + \nabla(g_{\lambda'} - g_{\lambda})(x).
\end{equation*}
It is easy to see that $\Phi(x) \to 0$ as $\lambda' \to \lambda$.
Since $v_{\lambda}$ has bounded third derivatives, we have $g_{\lambda}$ is globally Lipschitz.
Because $b_2 = -\nabla f$ is bounded and $\sigma_2 = g_{\lambda} I$ is bounded, Lipschitz and strict elliptic,
it follows from Theorem 3.1.2 in \cite{BKRS15} that
\begin{equation*}
\int_{\mathbb{R}^d } \bigg|\frac{\nabla \rho_{\lambda}(x)}{\rho_{\lambda}(x)}\bigg| \rho_\lambda(dx) \le \sqrt{\int_{\mathbb{R}^d } \bigg|\frac{\nabla \rho_{\lambda}(x)}{\rho_{\lambda}(x)}\bigg|^2 \rho_\lambda(dx)} < \infty.
\end{equation*}
By the dominated convergence theorem, we get $\int_{\mathbb{R}^d} |\Phi(x)| \rho_{\lambda}(dx) \to 0$ as $\lambda' \to \lambda$.
It suffices to apply Lemma \ref{lem:BRSest} to conclude.
\end{proof}

\quad The assumption \eqref{eq:dissipative} is a version of the dissipative condition, which is standard in Langevin sampling and optimization.
The assumption that $|\nabla f|$ is bounded restricts the range of the dissipative exponent $\kappa$ to $(0,1]$.
The only technical assumption in Theorem \ref{thm:stability} is that the solution $v_{\lambda}$ to the exploratory HJB equation \eqref{eq:ellipticPDE} is three times continuously differentiable with bounded third derivatives.
It implies that $\nabla^2 v_{\lambda}$ is continuously differentiable and is globally Lipschitz,
which is stronger than the result of Theorem \ref{T.2.7} that $\nabla^2 v$ is locally H\"older continuous.
It is interesting to know whether Assumption \ref{assump:1} (possibly with some additional conditions on $f$)
implies the boundedness of third derivatives of the solution to \eqref{eq:ellipticPDE}.

\section{Finite Time Horizon} 
\label{s5}

\quad The exploratory control problem \eqref{eq:expS}--\eqref{eq:regV} is a relaxed control problem in an infinite time horizon, and the associated exploratory HJB equation is, therefore,  elliptic.
Nevertheless, the arguments in the paper can be adapted, to the extent they can,  to the finite time setting where the HJB equation is parabolic.

\quad We follow the formulation in \cite{Zhou21}.
Fix $T > 0$, and consider the stochastic control problem whose value function  is
\begin{equation}
\label{eq:finiteclassical}
v(t,x) = \sup_{u \in \mathcal{A}_0(t,x)} \mathbb{E}\left[\int_t^T h_1(t,X^u_s, u_s) ds + h_2(X^u_T) \bigg| X^u_t = x\right],\;(t,x)\in [0,T]\times \mathbb{R}^d,
\end{equation}
where $h_1: [0,T]\times \mathbb{R}^d \times \mathcal{U} \to \mathbb{R}$ and $h_2: \mathbb{R}^d \to \mathbb{R}$ are reward functions, and $\mathcal{A}_0(t,x)$ is the set of admissible classical controls with respect to $X^u_t = x$. The state dynamics is
\begin{equation}
\label{eq:classicalS2}
dX^u_t = b(t,X^u_t, u_t) dt + \sigma(t,X^u_t, u_t) dB_t.
\end{equation}
Note here $b,\sigma,h_1$ depend on $t$ explicitly.

\quad Denote by $\partial_t$  the partial derivative  in $t$,  and by $\nabla_x$ and  $\nabla^2_x$ the gradient  and Hessian  in $x$ respectively.
The classical HJB equation associated with the problem \eqref{eq:finiteclassical}-- \eqref{eq:classicalS2} is
\begin{equation}
\label{eq:finiteHJB}
\left\{ \begin{array}{lcl}
\scaleto{\partial_t v(t,x) + \sup_{u \in \mathcal{U}} \left[h_1(t,x,u) + b(t,x,u) \cdot \nabla_x v(t,x) + \frac{1}{2}\tr(\sigma(t,x,u) \sigma(x,u)^T \nabla_x^2v(t,x))\right] = 0, \quad 0\leq t \le T,}{11 pt} \\
\scaleto{v(T,x) = h_2(x).}{9 pt}
\end{array}\right.
\end{equation}
It is known that a smooth solution to the HJB equation \eqref{eq:finiteHJB} gives the value function \eqref{eq:finiteclassical}.
The optimal control at time $t$ is $u_t^{*} = u^{*}(t,X^*_t)$, where $u^*:[0,T]\times \mathbb{R}^d \to \mathcal{U}$ is a deterministic mapping obtained by solving the ``$\sup_{u \in \mathcal{U}}$'' term in  \eqref{eq:finiteHJB}, and the optimally controlled process is governed by
\begin{equation}
\label{eq:optimalclassical2}
dX^{*}_t = b(t,X_t^{*}, u^{*}(t,X^{*}_t))dt +  \sigma(t,X_t^{*}, u^{*}(t,X^{*}_t))dB_t, \quad X_0^{*} = x,
\end{equation}
provided that it is well-posed.

\quad The exploratory control problem with finite time horizon is to solve an entropy-regularized relaxed control problem
whose value function is
\begin{equation}
\label{eq:finiteregV}
\scaleto{v_{\lambda}(t,x) = \sup_{\pi \in \mathcal{A}(t,x)} \mathbb{E}\bigg[ \int_t^T \bigg( \int_{\mathcal{U}} h_1(t,X^\pi_s, u) \pi_s(u) du - \lambda \int_{\mathcal{U}} \pi_s(u) \ln \pi_s(u) du  \bigg) ds + h_2(X^{\pi}_T)\bigg| X^\pi_t = x \bigg],
}{26 pt}
\end{equation}
where $\mathcal{A}(t,x)$ is the set of distributional control processes defined similarly to the infinite horizon setting, and the exploratory dynamics is
\begin{equation}
\label{eq:expS2}
dX^\pi_t = \widetilde{b}(t,X^\pi_t, \pi_t)dt +  \widetilde{\sigma}(t,X^\pi_t, \pi_t)dB_t,
\end{equation}
with
\begin{equation}
\label{eq:tildebsig2}
\widetilde{b}(t,x, \pi):=\int_{\mathcal{U}} b(t,x,u) \pi(u)du \quad \mbox{and} \quad \widetilde{\sigma}(t,x, \pi):=\bigg(\int_{\mathcal{U}} \sigma(t,x,u) \sigma(t,x,u)^T \pi(u)du\bigg)^{\frac{1}{2}}.
\end{equation}

\quad A similar argument as in Section \ref{s22} shows that the optimal feedback control at time $t$ is
\begin{equation}
\label{eq:regfeedbackfinite}
\scaleto{\pi^{*}(u,t,x) = \frac{\exp\left(\frac{1}{\lambda} \left[h(t,x,u) + b(t,x,u) \cdot \nabla_x v_{\lambda}(t,x) + \frac{1}{2}\tr(\sigma(t,x,u) \sigma(t,x,u)^T \nabla_x^2v_{\lambda}(t,x))\right] \right)}{\int_{\mathcal{U}} \exp\left(\frac{1}{\lambda} \left[h(t,x,u) + b(t,x,u) \cdot \nabla_x v_{\lambda}(t,x) + \frac{1}{2}\tr(\sigma(t,x,u) \sigma(t,x,u)^T \nabla_x^2v_{\lambda}(t,x))\right] \right) du},}{28 pt}
\end{equation}
the exploratory HJB equation is the following nonlinear parabolic PDE:
\begin{equation}
\label{eq:finiteellipticPDE}
\left\{ \begin{array}{lcl}
\scaleto{\partial_t v_{\lambda}(t,x) + \lambda \ln \int_{\mathcal{U}} \exp\bigg(\frac{1}{\lambda} \bigg[h(t,x,u) + b(t,x,u) \cdot \nabla_x v_{\lambda}(t,x) + \frac{1}{2}\tr(\sigma(t,x,u) \sigma(t,x,u)^T \nabla_x^2v_{\lambda}(t,x))\bigg] \bigg) du = 0, \quad 0\le t \le T,}{19.5 pt} \\
\scaleto{v_{\lambda}(T,x) = h_2(x),}{8.5 pt}
\end{array}\right.
\end{equation}
and the optimal state process is governed by
\begin{equation}
\label{eq:optimalreglam2}
dX^{\lambda,*}_t = \widetilde{b}(t,X^{\lambda,*}_t, \pi^{*}(\cdot, t,X^{\lambda,*}_t))dt +  \widetilde{\sigma}(t,X^{\lambda,*}_t, \pi^{*}(\cdot, t,X^{\lambda,*}_t))dB_t,
\end{equation}
provided that it is well-posed.

\quad To identify the value function \eqref{eq:finiteregV} (resp. \eqref{eq:finiteclassical}) as the solution to the HJB equation \eqref{eq:finiteellipticPDE} (resp. \eqref{eq:finiteHJB}),
the verification theorem requires that these solutions to be $\mathcal{C}^{1,2}_{t,x}$.
However, when $\mathcal{U} = [\frac12,1]$, $h_1(x,u) = 0$, $b(x,u) = 0$ and $\sigma(x,u) = \sqrt{2u}I$,
a result of \cite{CS08} shows that the solutions to \eqref{eq:finiteHJB} are not $\mathcal{C}_{t,x}^{1,2}$. For general fully nonlinear parabolic PDEs, the solution is only known to be $\mathcal{C}_{t,x}^{\alpha, 1+\alpha}$ for some $\alpha \in (0,1)$.
We record this fact in the following theorem.
\begin{theorem} \label{thm:paraboliceqn}
Let Assumption \ref{assump:hbsig} hold for $h_1(\cdot,\cdot,\cdot), b(\cdot,\cdot, \cdot), \sigma(\cdot,\cdot,\cdot)$, and assume that $h_2(\cdot)$ satisfies
\[
|h_2(\cdot)| \leq \gamma_r\quad\text{ in }B_r.
\]
Then the HJB equation \eqref{eq:finiteellipticPDE} (resp. \eqref{eq:finiteHJB}) has a unique solution $v_\lambda$ (resp. $v$) of sub-quadratic growth for $t \in [0, T]$.
Moreover,
\begin{enumerate}[itemsep = 3 pt]
\item[(i)]
$v_{\lambda}, v$ are $\mathcal{C}_{t,x}^{\alpha,1+\alpha}$  locally uniformly in $[0,T)\times\bbR^d$ for some $\alpha \in (0,1)$.
\item[(ii)]
There exists $C > 0$ such that $\sup_{x\in B_r,t \in [0,T]} (|v(t,x)| + |v_{\lambda}(t,x)|) \le C \gamma_r$ for each $r \ge 1$.
\item[(iii)]
$v_{\lambda} \to v$ locally uniformly as $\lambda \to 0^{+}$.
\end{enumerate}
\end{theorem}

\quad We refer to \cite{LW1, LW2} and \cite{CKS00} for the interior point-wise regularity estimate for fully nonlinear parabolic PDEs.
To close the gap between what the verification theorem requires for the regularity of HJB equations and
what is known for the regularity of general fully nonlinear parabolic PDEs,
it remains a significant open question to find proper assumptions on $h_1(\cdot,\cdot,\cdot), b(\cdot,\cdot, \cdot), \sigma(\cdot,\cdot,\cdot)$ under which the HJB equations \eqref{eq:finiteellipticPDE} and  \eqref{eq:finiteHJB} have  unique $\mathcal{C}^{1,2}_{t,x}$ solutions?


\quad On the other hand, under a further assumption that $\sigma$ does not depend on $u$, \cite{GR06, GR062} showed that the verification theorem only requires the solution to the HJB equation to be $\mathcal{C}^{0,1}_{t,x}$.
Combining this result with Theorem \ref{thm:paraboliceqn}, we obtain the following analog of Theorem \ref{thm:main0} for the exploratory control problem with a finite time horizon.

\begin{theorem}
Consider the exploratory control problem \eqref{eq:finiteregV}--\eqref{eq:expS2}  whose value function is $v_\lambda$.
Let Assumption \ref{assump:hbsig} hold for $h_1(\cdot,\cdot,\cdot), b(\cdot,\cdot, \cdot), \sigma(\cdot,\cdot,\cdot)$ with $\sigma(\cdot,\cdot,\cdot) \equiv \sigma$ being constant, and assume that $h_2(\cdot)$ satisfies
\[
|h_2(\cdot)| \leq \gamma_r\quad\text{ in }B_r.
\]
Further assume that the SDE \eqref{eq:optimalreglam2} is well-posed.
Then $v_{\lambda}$ is the unique solution of sub-quadratic growth to the exploratory HJB equation \eqref{eq:finiteellipticPDE}.
Moreover, $v_{\lambda}$ is locally $\mathcal{C}^{\alpha,1+\alpha}_{t,x}$ for some $\alpha \in (0,1)$, and
\[
v_{\lambda} \to v \quad \mbox{locally uniformly as } \lambda \to 0^+,
\]
where $v$ is the value function of the classical control problem \eqref{eq:finiteclassical}
--\eqref{eq:classicalS2}
and the unique solution of sub-quadratic growth to the classical HJB equation \eqref{eq:finiteHJB}.
\end{theorem}

\section{Conclusions}
\label{s6}

\quad In this paper, we study the exploratory HJB equation arising from a continuous-time reinforcement learning framework -- that of the exploratory control -- put forth by  \cite{WZZ20}.
We establish the well-posedness and regularity of its solution under general assumptions on the system dynamics  parameters.
This allows for identifying the value function of the exploratory control problem in general cases, which goes beyond the LQ setting.
We also establish a connection between the exploratory control problem and the classical stochastic control problem by showing that the value function of the former converges to that of the latter as the weight parameter for exploration tends to zero.
We then apply our general theory to a special example -- the exploratory temperature control problem originally introduced by \cite{GXZ20} as a variant of SA.
We provide a detailed analysis of the problem, with an explicit rate of convergence derived  as the weight parameter vanishes.
We also consider the long time behavior of the associated optimally controlled process, and study properties of its stationary distribution.
The tools that we develop in this paper encompass stochastic control theory, partial differential equations and probability theory.

\quad There are many important open, if technical, questions. First, we have proved in Theorem \ref{lem:generalkey} that the exploratory HJB equation \eqref{eq:regellipticPDE} has a unique smooth solution if Assumption \ref{assump:hbsig} holds.
In particular, the drift $b$ is allowed to have sub-linear growth in $x$.
However, in order to identify the value function of the exploratory control problem as the solution to \eqref{eq:regellipticPDE}, one needs  the SDE \eqref{eq:optimalreglam} to be well-posed.
This is satisfied if $b$ is bounded. The question now is
what assumptions, in addition to Assumption \ref{assump:hbsig} and in particular the sub-linear growth of $b(\cdot, u)$, are required to ensure the well-posedness of  \eqref{eq:optimalreglam}. A related question is whether we have the well-posedness of \eqref{eq:regellipticPDE} and \eqref{eq:HJBclassical} for $b(\cdot, u)$ beyond sub-linear growth (e.g. of linear growth or polynomial growth).
If the answer to the first question is positive, then we will have a complete characterization of the exploratory control problem for sub-linear $b$'s.
In the case of the exploratory temperature control problem, we will then no longer need to impose a bounded restriction on $b = |\nabla f|$.
As discussed after Theorem \ref{thm:stationary}, we may then specify a convergence rate for the optimally controlled process with a Lyapunov condition.

\quad  Second, in the study of  stability of stationary distributions of the optimal state processes with different $\lambda$'s (Theorem \ref{thm:stability}), we make a technical assumption that $v_{\lambda}$ has bounded third derivatives.
It is challenging, yet interesting, to know under what conditions on $f$ this assumption holds.
It can be shown using the arguments in Section \ref{s3} that for bounded $f$ with all bounded derivatives, the solution to the exploratory HJB equation \eqref{eq:ellipticPDE} has such a solution.
To completely solve the stability problem, a first step is to prove/disprove whether such a solution exists for $f$ of linear growth.
More generally, one can ask under what conditions on $h(\cdot,\cdot, \cdot), b(\cdot,\cdot, \cdot), \sigma(\cdot,\cdot, \cdot)$ does  \eqref{eq:regellipticPDE} have a unique (viscosity) solution which is $\mathcal{C}^3$ with bounded third derivatives.
As discussed after Theorem \ref{thm:main0}, if  \eqref{eq:regellipticPDE} has a unique $\mathcal{C}^3$ solution,
then the SDE \eqref{eq:optimalreglam} is well-posed under additional non-explosion conditions.

\bigskip
{\bf Acknowledgement:}
Tang gratefully acknowledges financial support through an NSF grant DMS-2113779 and through a start-up grant at Columbia University.
Zhou gratefully acknowledges financial supports through a start-up grant at Columbia University and through the Nie Center for Intelligent Asset Management.

\bibliographystyle{abbrvnat}
\bibliography{unique}
\end{document}